\documentclass[10pt] %reqno] 
{amsart} 
\usepackage{amsmath, latexsym, amsthm, amssymb, mathrsfs, amsfonts, enumerate}
\usepackage{enumitem}
\usepackage{mathtools}
\usepackage[utf8]{inputenc}
\usepackage{tikz, %pgfplots, 
subfigure, array, tabularx}
\usepackage{multicol, multirow}
\usetikzlibrary{arrows,automata}
\usepackage{varwidth, xparse}
\usepackage{comment}
\usepackage{kotex}
\usepackage{accents}

\usepackage%[%hidelinks, backref=page]
{hyperref}
\hypersetup{
    colorlinks=true,
    linkcolor=blue,
    filecolor=magenta,      
    urlcolor=cyan,
    citecolor=red,
}
\usepackage{todonotes}
\usepackage{ulem}

\theoremstyle{plain}
\newtheorem{thm}{Theorem}[section]
\newtheorem{cor}[thm]{Corollary}
\newtheorem{lem}[thm]{Lemma}
\newtheorem{prop}[thm]{Proposition}

\newtheorem{rmk}[thm]{Remark}  
\newtheorem{defn}[thm]{Definition}
\numberwithin{equation}{section}

\definecolor{ao}{rgb}{0, 0.5, 0}

\newcommand{\R}{\mathbb R}
\newcommand{\Z}{\mathbb Z}

\newcommand{\dist} {\text{dist\! }}

\DeclareMathOperator{\supp}{supp}
\DeclareMathOperator{\rank}{rank}

\newcommand{\fq}{\mathfrak q}
\newcommand{\fm}{\mathfrak{q'}}
\newcommand{\fmm}{\mathfrak{q''}}

\title[]
{Local smoothing and maximal estimates for average over surfaces of codimension 2 in $\mathbb R^4$}

\author{Seheon Ham} \author{Hyerim Ko} 
 
\address[Seheon Ham]{Department of Mathematical Sciences, Seoul National University, Seoul 08826, Republic of Korea}
\email{seheonham@snu.ac.kr} 

\address[Hyerim Ko]{Department of Mathematics, and Institute of Pure and Applied Mathematics, Jeonbuk National University, Jeonju, 54896, Republic of Korea}
\email{kohr@jbnu.ac.kr}
\thanks{}
\keywords{local smoothing, maximal function, submanifold}
\subjclass[2020]{42B25}

\begin{document}

\begin{abstract} 
In this paper, we obtain local smoothing  
estimates for the averages over nondegenerate surfaces of codimension $2$ in $\mathbb R^4$. 
We make use of multilinear restriction estimates and decoupling inequalities for a hypersurface in $\mathbb R^5$, a conical extension of a two-dimensional nondegenerate surface along two flat directions. 
We also establish sharp $L^p$--$L^q$ estimates for maximal averages over nondegenerate surfaces of half the ambient dimension in $\mathbb R^{2n}$ for even $n \ge 2$.  
\end{abstract}

\maketitle

\section{Introduction}
Let $d \ge 2$. For $1 \le k \le d-1$, let $S$ be a smooth compact $k$-dimensional submanifold in $\mathbb R^d$.
We consider the averaging operator $\mathcal A$ associated with dilations of $S$ defined by  
\begin{equation*} 
\mathcal   A_S f (x, t )  = \chi_I(t)  \int_{S} f(x - t y) d\sigma(y)  ,
\end{equation*}
where $\chi_I$ is a smooth function supported in $I = [1/2,4]$ and $d\sigma$ is the surface measure on $S$. 

\subsection*{Local smoothing estimates}
The smoothing properties of $\mathcal A_S$ have been extensively studied under a suitable curvature condition on $S$.
For fixed $t \neq 0$, there are various model 
cases for which $\| \mathcal A_S f(\cdot,t)\|_{p} \lesssim \|f\|_{L_\sigma^p}$ for some  $\sigma<0$ under the condition that $S$ is nondegenerate (e.g., see  \cite{Pe, M, SSS, OSS, PS2, BGHS2, KLO23}). 
Here, the best possible  regularity exponent is $\sigma=-k/p$.

An additional integration over $t$ yields an extra regularity gain for a certain range of $p>2$. This phenomenon is  known as  local smoothing:  
\begin{equation}\label{submanifold}
\| \mathcal A_S f \|_{L^p (\mathbb R^{d+1})} \lesssim \|f\|_{L^p_\sigma(\mathbb R^d)}, ~\quad \sigma > - (k+1)/p   .
\end{equation} 
Here, the regularity condition $\sigma > - (k+1)/p$ is known to be sharp. 

When $d=2$ (i.e., $k=1$) and $S$ is nondegenerate and strictly convex, \eqref{submanifold} is equivalent to the local smoothing phenomenon for the wave propagator, first observed by Sogge \cite{So91}. 
The conjecture was proved by Guth--Wang--Zhang \cite{GWZ} (for earlier results, see \cite{Wo00, MSS, TV, LV, BD0, JLee} and references therein). 
In fact, they obtained a sharp square function estimate for the cone in $\mathbb R^3$, which implies the conjectured local smoothing estimate.

In higher dimensions for  $k=d-1\ge 2$, the conjecture remains widely open.
Analogous estimates have been established for surfaces with non-vanishing Gaussian curvature when the principal curvatures have different signs, or for variable coefficient generalization. 
We refer to \cite{LW02, BD, BHS, BS, GLMX, GW} for some partial results.

For $k=1$ and $d \ge 3$, when $S$ is a smooth curve with nonvanishing torsion, \eqref{submanifold} was recently established by Gan--Maldague--Oh \cite{GMO} (see \cite{PS2, BGHS, KLO23, BH} for some previous results and sharpness).

In contrast, the intermediate cases of $1< k<d-1$ (i.e. submanifolds with dimension $k\ge2$ and codimension $d-k\ge2$) remain largely open. 
To the best of our knowledge, no general local smoothing property for averaging operator associated with such submanifolds has been studied, even for fundamental model cases.
In this paper, we obtain sharp $L^p$--$L^q$ local smoothing estimates (in the regularity exponent) for two-dimensional surfaces in $\mathbb{R}^4$ under a suitable nondegeneracy condition.

\bigskip

We consider a smooth compact surface $\Gamma$ given by
\[
 \Gamma(u,v) = (u,v,\phi_1(u,v),\phi_2(u,v)) : [-1,1]^2 \to \mathbb R^4  ,
\]
where $\phi_i$, $i=1,2$, are smooth real valued function with $\phi_i(0,0) =\nabla \phi_i (0,0) =  0$. 
We say that $\Gamma$ is \textit{nondegenerate} if 
\begin{equation}\label{2by3rank}
\rank \begin{pmatrix} \partial_{u}^2 \phi_1 &\partial_{uv}^2 \phi_1 & \partial_{v}^2 \phi_1 \\ \partial_{u}^2 \phi_2 & \partial_{uv}^2 \phi_2 & \partial_{v}^2 \phi_2 \end{pmatrix}  = 2
\end{equation}
holds on the domain $[-1,1]^2$.

Two fundamental examples are $\Gamma_\circ (u,v) = (u,v,u^2-v^2, 2uv)$ and $\Gamma_1 (u,v) = (u,v,u^2,v^2)$.
The key difference between $\Gamma_\circ$ and $\Gamma_1$ is the rank of Hessian matrices $\nabla_{u,v}^2 \phi_i$, $i=1,2$. 
In the case of $\Gamma_\circ$, the Hessian matrices are of full rank, which can be generally characterized by the condition 
\begin{align}\label{G-nonzero}
\det \big( \theta_1 \nabla_{u,v}^2 \phi_1(u,v) + \theta_2 \nabla_{u,v}^2 \phi_2(u,v) \big  ) \neq0
\quad \text{for all} ~  (\theta_1, \theta_2) \in \mathbb S^1.
\end{align}
(See \cite{Ch82, Ch85, Mo96, Ba02}.)
This condition allows us to apply the method of stationary phase, which in turn enables us to express the averaging operator as a sum of associated Fourier extension operators.

On the other hand, for $\Gamma_1$, there exists a direction along which \eqref{G-nonzero} fails. 
Because of this degeneracy, the implicit function theorem cannot be applied directly, and $L^p$--$L^q$ boundedness can depend on the degree of the degeneracy. 
In Theorem \ref{thm:rank1}, we consider a class of surfaces that generalizes $\Gamma_1$. 

\bigskip

Now, we define the averaging operator for $\Gamma$ by
\begin{align*}
\mathcal A_\Gamma f(x,  t) = \chi_I (t) \iint f(x -  t \,\Gamma (u, v) ) \varphi(u,v)du dv ,
\end{align*}
where  
$\varphi\in C_0^\infty$ is a smooth function supported on $[-1,1]^2$ and $I = [1,2]$. 
Our first main result is a sharp $L^p$--$L^q$ local smoothing estimate for $\mathcal A_\Gamma$ with $\Gamma=\Gamma_\circ$. 

\begin{thm}\label{LS}
Let $\Gamma = (u,v, u^2-v^2, 2uv)$.  
If $\frac 1 q \le \min\{ \frac 1p,\, \frac12(1-\frac1p),\,  \frac 1{3p}+\frac16 \}$,
then 
\begin{equation}\label{est-LS}
\|\mathcal A_\Gamma f \| _{L^q_{x,t}(\mathbb R^5)} \le C \|f \|_{L^p_\gamma (\mathbb R^4)}
\end{equation}
holds for $\gamma > \frac 2p -\frac 5 q$. 
\end{thm}

The regularity $\gamma > \frac 2p -\frac 5 q$ is optimal when $\frac 1 q \le \min\{ \frac 1p,\, \frac12(1-\frac1p)\}$, which will be discussed in Section \ref{sec:sharpness}.
However, the range of $1/p,1/q$ in Theorem \ref{LS} is not sharp.

To prove Theorem \ref{LS}, we use an induction on scale argument (see \cite{BG, HL, HKL2}, for example), combined with an $L^2$--$L^3$ trilinear restriction estimate and an $L^4$ decoupling inequality. (Therefore, Theorem \ref{LS} also holds for perturbations of $\Gamma$.)  
In particular, we reduce $\mathcal A_\Gamma$ to the Fourier extension operator associated with a hypersurface $\Sigma = \Sigma_\Gamma$ in $\mathbb R^5$, which is homogeneous of degree 1 and has at least two nonzero principal curvatures. 
In fact, $\Sigma$ is obtained by a conical extension of a two-dimensional nondegenerate surface along two flat directions. In addition, its normal vector takes the form $(\Gamma(u,v),-1)$ and the transversality of three normal vectors of $\Sigma$ is guaranteed whenever their supports are separated (see Lemma \ref{transv}).
Then we apply the multilinear restriction estimate with a low level of multilinearity due to Bennett--Carbery--Tao \cite{BCT}, which is $L^2$--$L^3$ trilinear restriction estimate.
On the other hand, we make use of the $L^4$ decoupling inequality for hypersurfaces in $\mathbb R^3$ with nonvanishing Gaussian curvature established in \cite{BD0,BD}, to which the decoupling inequality for $\Sigma$ can be reduced via a truncation argument.

Since we treat averaging over dilations of the two-dimensional surface $\Gamma$, it is plausible that the same argument as in our proof of Theorem \ref{LS} would provide a partial result toward local smoothing estimates for the wave operator in $\mathbb R^{3+1}$  (see \cite{BS, GLMX, GW} for recent progress). 
In contrast to the case of wave operator in $\mathbb R^{3+1}$, the hypersurface $\Sigma$ behaves as a complex conical extension (two flat directions) of the complex curve $\Gamma(z) = (z,z^2)$  other than a  usual conical extension of the sphere in $\mathbb R^3$. 

To obtain the sharp range, one may need to prove the local smoothing estimate at $p=q=3$ for $\Sigma$ (cf. 
$L^3$ local smoothing conjecture for the wave equation in $\mathbb R^{3+1}$). However, this problem appears to involve technical difficulties, and we do not pursue it here.  

For $\Gamma = \Gamma_1$, the local smoothing estimate \eqref{est-LS} can be obtained by repeatedly applying the $L^4$ local smoothing estimate in \cite{GWZ} combined with Fubini's theorem. In fact, the argument works for a slightly generalized version of $\Gamma_1$. This will be discussed when proving corresponding maximal estimate in Theorem \ref{thm:rank1}.

\bigskip

\subsection*{Maximal estimates}

Now, we define the maximal function for a smooth compact $k$-dimensional manifold $S$  by
\begin{equation*} 
\mathcal  M_S f (x, t )  = \sup_{t\in J}  \left| \mathcal A_S f(x,t)  \right|,
\end{equation*}
where  $J$ is an interval $\subseteq (0,\infty)$.   

We study the mapping properties of the maximal operator $\mathcal M_S$ under suitable curvature conditions on $S$.
For $k = d-1$, the most fundamental example of $S$ is the sphere. In this case, 
Stein \cite{St76} ($n\ge 3$) and Bourgain \cite{B} ($d=2$) proved that $\mathcal M_S$ with $J = (0,\infty)$ is bounded on $L^p$ 
if and only if $\frac d{d-1}< p \le \infty$.
When $J $ is a compact interval, the localized maximal operator $\mathcal M_S$ has the $L^p$ improving property, that is, $\mathcal M$ maps $L^p$ boundedly to $L^q$ for some $q > p$. 
Except for a few endpoint cases, $L^p$--$L^q$ estimates for $\mathcal M$ were completely characterized in \cite{Schlag, SS, Lee}. Recently, there have been several works on generalization of the spherical maximal function, including studies on the elliptic maximal function and the strong spherical maximal function (e.g. see \cite{LLO1, LLO2, CGY, HZ}  and references therein).

In connection with geometric measure theory, different forms of the spherical maximal function have been developed by \cite{Wo97, KW, Za12}.  
See also \cite{PYZ, Za23} for further results on 
Wolff's maximal average associated with families of curves in the plane.  
We refer to \cite{HKL2, CDK} for circular maximal estimates and local smoothing estimates for the wave operator relative to fractal measures.

For $k=1$, when $S$ is a smooth curve with nonvanishing torsion, Ko--Lee--Oh \cite{KLO22} and Beltran--Guo--Hickman--Seeger \cite{BGHS} proved that 
$\mathcal M$ is bounded in $L^p(\mathbb R^3)$ if and only if $p>3$. (See also \cite{PS2, BD0} for some previous results.) 
In higher dimensions $d\ge4$, after the pioneering work  \cite{KLO23}, Gan-Maldague-Oh \cite{GMO} recently obtained $L^p$ boundedness of $\mathcal M$ for $p > 2(d-2)$ from the corresponding local smoothing estimate. This result is sharp when $d=4$ (i.e., for $p>4$). 
The sharp $L^p$ improving property of $\mathcal M$ for nondegenerate curves in $\mathbb R^3$ was settled by Beltran--Duncan--Hickman \cite{BDH}. 

For the intermediate cases i.e. $ 1<  k < d-1$, Rubio de Francia \cite{Ru86} established $L^p$ boundedness for $p >\frac{2a+1}{2a} $ assuming the decay condition $ |\widehat{d\sigma}(\xi)| \le C |\xi|^{-a}$ with $a>1/2$.
However, sharp characterization of the $L^p$ improving property for $\mathcal M$ remains open in general.

Our next main result proves sharp $L^p$--$L^q$ estimates for the maximal operator $\mathcal M$ associated with certain classes of $\Gamma $ satisfying the nondegeneracy condition \eqref{2by3rank}, including two model examples $\Gamma_\circ$ and $\Gamma_1$. 
We further extend these results to higher even dimensions, providing sharp estimates for surfaces of half the ambient dimension beyond the previously known cases of $k=1$ or $k=d-1$.

\bigskip 

For a general $n$-dimensional framework, $n\ge1$, we set
\[
\mathcal W_n:=\Big\{\Big(\frac1p,\frac1q):\,
\frac 1q\le \frac1p \le \frac 2q,\quad
\frac{3n}{2p} \le \frac n2+\frac{3n-2}{2q}, \quad
\frac {2n}p \le n+\frac{n-1}q\Big\},
\]
which is the convex hull of $\mathrm O=(0,0)$, 
\[
\mathrm P_1 =\Big(\frac{2n}{3n+2},\frac n{3n+2}\Big), \quad
\mathrm P_2 = \Big(\frac{2n-1}{3n-1},\frac{n}{3n-1}\Big),\quad
\mathrm P_3 = \Big(\frac n{n+1},\frac {n}{n+1}\Big).
\] 

Consider $n$-dimensional submanifolds of $\mathbb R^{2n}$ parametrized by $\mathbf u=(u_1,\dots,u_n)\in [-1,1]^n \rightarrow \Phi (\mathbf u)\in\R^n$ with
\begin{align*} 
\Phi(\mathbf u)=\big(\phi_1(\mathbf u),\dots,\phi_n(\mathbf u) \big), \quad \phi_k\in C_0^\infty(\mathbb R^n)  .
\end{align*}

Extending \eqref{2by3rank}, we say that $\Phi$ is nondegenerate if the $n \times \frac{n(n+1)}2$ matrix consisting of all second derivatives of $\phi_i$ has full rank, and we assume the curvature condition
\begin{align}\label{curv-high}
\det\Big(\sum_{i=1}^n \theta_i \nabla_{\mathbf u}^2 \phi_i(\mathbf u) \Big) \neq 0,
\quad \text{for all} ~ ( \theta_1,\dots,\theta_n )\in \mathbb S^{n-1},
\end{align}
which coincides with \eqref{G-nonzero} when $n=2$.

For odd $n$, the condition \eqref{curv-high} fails in general. For quadratic $\phi_i$'s, the left-hand side of \eqref{curv-high}  is a homogeneous polynomial of $\theta$ of odd degree $n$, which must have a zero.

Let $\Gamma(\mathbf u) = (\mathbf u,  \Phi(\mathbf u))$. We define the maximal function 
\[ 
\mathcal M_\Gamma f(x,t)=\sup_{t \in I } \left| \int f(x-t\Gamma(\mathbf u) )\varphi(\mathbf u)\,d\mathbf u \right|,
\]
where  $\varphi\in C_0^\infty$ is a smooth function supported on $[-1,1]^n$.
The $L^p$ improving property of $\mathcal M_\Gamma$ can be characterized as follows. 

\begin{thm}\label{n-dim-thm}
Let $n\ge 2 $ be even.
Suppose that $\Gamma$ is nondegenerate and  satisfies \eqref{curv-high}.
Then 
\begin{align}\label{supA}
\big\| \mathcal M_\Gamma f\big\|_{L^q(\mathbb R^{2n})} \le  C\|f\|_{L^p(\mathbb R^{2n})}, 
\end{align}
whenever \[ \, (1/p,1/q)\in \begin{cases}   \mathcal W_n \setminus\{\mathrm P_1, \mathrm P_2, \mathrm P_3\}, &~  n \ge 4  \\   \mathcal W_2 \setminus  \{ (\mathrm O, \mathrm P_1], (\mathrm P_1, \mathrm P_2] \}, &~  n=2\end{cases}.\]
Moreover, the restricted weak-type estimate $L^{p,1}(\R^n) \to L^{q,\infty}(\R^{2n})$ 
holds at the endpoints $(1/p,1/q) = \mathrm P_1, \mathrm P_2, \mathrm P_3$ when $n\ge4$, and at $\mathrm P_2,\mathrm P_3$ when $n=2$.
\end{thm}

The range of $p,q$ in Theorem \ref{n-dim-thm} is sharp in the sense that \eqref{supA} fails outside of $\mathcal W_n$ (see Section \ref{sec:sharpness}). 

To prove Theorem \ref{n-dim-thm}, we follow the argument in \cite{Lee} for the spherical maximal estimates, and therefore the strong type bounds at some endpoints remain open. 
Note that Theorem \ref{LS} implies \eqref{supA} only for $(1/p,1/q)$ are contained in $\mathcal W_2\setminus\{(O,P_1],(P_1,P_2],(P_2,P_3]\}$.

The condition \eqref{curv-high} allows us to successfully  reduce the averaging operator to Fourier extension operator associated with a hypersurface with at least $n$ nonzero principal curvatures, via the method of stationary phase. 
Then we apply Greenleaf's $L^2$--$L^{\frac{2(n+2)}{n}}$ restriction estimate.
In the case of $n=2$, this matches with  $L^2$--$L^4$ restriction estimate for the paraboloid in $\mathbb R^3$ or the cone in $\mathbb R^4$.
Combined with $L^2$ decay and trivial estimates, we invoke Bourgain's interpolation trick (as in \cite{Lee}) to obtain restricted weak type bounds at $\mathrm P_1, \mathrm P_2, \mathrm P_3$, and then interpolate with $\mathrm O$.
When $n=2$, this argument does not provide the restricted weak type at the endpoint $\mathrm P_1$, where   Greenleaf's restriction theorem holds.

 \bigskip

We now turn to the case $\Gamma_1(u,v) = (u,v,u^2,v^2)$, where the condition  \eqref{G-nonzero} fails in a certain direction.  
Under this vanishing condition, we generalize $\Gamma_1$  to $\Gamma= (u,v,\phi_1(u,v), \phi_2(u,v))$, where at least one of $\phi_1$ or $\phi_2$ depends only on a single variable.  Without loss of generality, we assume that $\phi_2(u,v) = \phi_2(v)$. 
Since $\Gamma$ satisfies the nondegeneracy condition \eqref{2by3rank}, we further assume that  $\partial_u^2\phi_1\neq 0$ and $\partial_v^2\phi_2\neq0$.

\begin{thm}\label{thm:rank1}
Let $\Gamma=(u,v,\phi_1(u,v),\phi_2(v))$ be a smooth nondegenerate surface such that $\partial_u^2\phi_1\neq 0$ and $\partial_v^2\phi_2\neq0$.
For $(1/p,1/q)\in \mathcal  W_1 \setminus \big\{ ( 2/5,1/5), (1/2,1/2)\big\}$, 
\begin{equation}\label{est:rank1}
\big\| \mathcal M_\Gamma f\big\|_{L^q(\mathbb R^{4})} \le  C\|f\|_{L^p(\mathbb R^{4})}, 
\end{equation}
Moreover, the restricted weak type estimate holds at the endpoint $( 2/5,1/5)$.
\end{thm}

Here, $\mathcal W_1$ is the triangle with vertices $\mathrm O = (0,0)$, $\mathrm P_1=(2/5,1/5)$, and $\mathrm P_2=\mathrm P_3=(1/2,1/2)$, which coincides with the one for the circular maximal function.
In Proposition \ref{nece:rank1}, we show $\mathcal W_1$ is also optimal for \eqref{est:rank1} to hold. 

To prove the restricted weak type estimate in Theorem \ref{thm:rank1}, we rely on the local smoothing estimate for the wave operator in $\mathbb R^{2+1}$ established by Lee \cite{Lee}. 
The result of Guth--Wang--Zhang \cite{GWZ} also yields \eqref{est:rank1} up to endpoints.

\begin{rmk}\label{rmk:rank1} 
If both $\phi_1$ and $\phi_2$ involve mixed variables, the boundedness of the maximal function   depends on the degree of degeneracy.  
For instance, for  $(\phi_1,\phi_2)=(u^2 + \varepsilon v^k ,v^2 + \varepsilon u^k)$ with small $\varepsilon \neq 0$, a direct application of van der Corput's lemma yields an $L_{-(\frac 1 2+\frac 1k)}^2-L_{x,t}^2$ estimate for $\mathcal Af$, which in turn leads to $L^p$ maximal bounds for some $p<2$.
For polynomial $\phi_1,\phi_2$, the Fourier decay is determined by the Newton polyhedron; see \cite{Ba02}.
\end{rmk}

\subsection*{A geometric application}
Following the argument discussed in \cite[Proposition 1.5]{HKLO}, a lower bound of the Hausdorff dimension of a union of geometric objects is determined by a smoothing order of local smoothing estimate for the associated averaging operator.  
Namely, by applying Theorem 1.3, we can derive a lower bound for the Hausdorff dimension of the union of translations and dilations of $\Gamma_\circ$. 
The following corollary is a consequence of Theorem \ref{n-dim-thm} (see the proof of Theorem 1.1 in \cite{HKLO}), so we omit the proof. 

\begin{cor}
Let $E$ be a Borel set in $\R^5$ such that $\dim_H E>\alpha$  for some $0<\alpha \le2$. 
If $F$ is a Borel set in $\R^4$ containing $\bigcup_{(x,t)\in E}(x+t\Gamma_\circ)$, Then $\dim_H F \ge \alpha+2$. Especially, if $E$ has Hausdorff dimension greater than $2$, then 
$F$ has positive Lebesgue measure.
\end{cor}

\subsection*{Organization of the paper}
Through Section \ref{sec:pre}--\ref{sec:completion}, we prove Theorem \ref{LS}.  
Our argument relies on trilinear restriction estimates and a sharp decoupling inequality, which  are presented  in Section \ref{TRI} and Section \ref{DEC}, respectively.
Proof of Theorem  \ref{n-dim-thm} will be given in 
Section \ref{sec:proof of high}. 
We discuss the sharpness of Theorems \ref{LS}, \ref{n-dim-thm}, and \ref{thm:rank1} in Section \ref{sec:sharpness}.
In Appendix \ref{rank1case}, we prove Theorem \ref{thm:rank1}.

\section{Proof of Theorem \ref{LS}: Preliminaries}\label{sec:pre}

In this section,  
we prove the local smoothing estimate \eqref{est-LS} for a slightly more general $\Gamma$ than $\Gamma_\circ = (u,v,u^2-v^2, 2uv)$. 
Indeed, we further assume that $(\phi_1,\phi_2)$ satisfies the Cauchy-Riemann equation.
Since $\Gamma_\circ$ is equivalent to the complex curve $(z,z^2)$, the surface $\Gamma$ considered in Theorem \ref{LS-com} can be regarded as a complex curve $(z, \Phi (z))$ for a holomorphic function $\Phi = \phi_1 + i \phi_2 $ satisfying $\Phi''(z) \neq 0$ for all $|z| < 1 $.

Our approach relies on the induction on scale argument combined with multilinear restriction estimates and a decoupling inequality. 
We adapt the argument used in \cite{KLO23}, which was originally introduced by Bourgain--Guth \cite{BG}. We decompose the domain of $\Gamma$, separating the linear part and the multilinear part. For the linear part, we rescale the operator with a small support of $\Gamma$ and apply the induction assumption. 
For the multilinear part, we derive estimates for a wider range of $p,q$ under a condition of separated domains. 
By efficiently combining these two parts, we obtain the desired estimates.

\begin{thm}\label{LS-com}
Let $\Gamma$ be a nondegenerate surface satisfying \eqref{G-nonzero}. 
Suppose that $(\phi_1,\phi_2)$ satisfies the Cauchy-Riemann equation: 
\begin{equation}\label{holo}
\partial_u  \phi_1 = \partial_v \phi_2 ~\text{ and }~\partial_v \phi_1 = -\partial_u \phi_2 .
\end{equation}
If $\frac 1 q \le \min\{ \frac 1p,\, \frac 1{3p}+\frac16,\, \frac12(1-\frac1p)\}$,
then \eqref{est-LS} holds for $\gamma > \frac 2p -\frac 5 q$. 
\end{thm}

Let $B^d(x,r)$ be the ball of radius $r>0$ centered at $x\in \R^d$.
Let $\beta_0$ be a real radial bump function supported on $B^4(0,2)$ that satisfies $\beta_0=1$ on $B^4(0,1)$. 
We set $\beta_\lambda(\cdot)=\beta_0(\lambda^{-1}\cdot)-\beta_0( 2\lambda^{-1}\cdot)$ for $\lambda>0$.
Let $P_\lambda$ be the standard Littlewood-Paley projection operator given by $\widehat{P_\lambda f}(\xi) = \beta_\lambda (\xi)\widehat f(\xi)$, $\lambda \ge1$ and $\widehat{P_0f}(\xi)=\beta_0(\xi)\widehat f(\xi)$.

Since $\sup_{t \in I } |\mathcal A_\Gamma P_0(x,t) |$ is bounded by the Hardy--Littlewood maximal function and Young's inequality, we see $\| \sup_{t \in I } |\mathcal A_\Gamma P_0 f(\cdot,t)| \|_{L^q} \lesssim \|P_0 f\|_{L^q}\lesssim  \|f\|_{L^p}$ for $1\le p \le q$. 
Therefore, 
it is sufficient to prove the contribution 
from the case of  $\lambda \ge1$.  
Using the Fourier inversion formula, we consider the following Fourier integral operator $ \mathcal A_\lambda  [\varphi] 
=\mathcal A_\Gamma [\varphi] P_\lambda $ defined by 
\[  
\mathcal A_\lambda [\varphi]f(x,t) = \chi_I(t)\int e^{ix\cdot\xi} \int e^{-it\Gamma(u,v)\cdot \xi}  \varphi  ( u,v) du dv \beta_\lambda(\xi) \widehat{f} (\xi) d\xi .
\]

Let   
\[
\xi=(\xi_1,\xi_2, \xi_3,\xi_4)=: (\xi',\xi'') \in \mathbb R^2\times \mathbb R^2.
\] 
Also, let $c_* =  \| J_\Phi^\intercal \|$   for the Jacobian matrix $J_\Phi$ with $\Phi = (\phi_1,\phi_2)$. Here, $\|\mathrm M\|=\sup_{\|x\|=1}\big| \mathrm Mx\big|$ for a matrix $\mathrm M$, and $M^\intercal $ is the transpose of $M$.
We can write
\[ 
\mathcal A_\lambda[\varphi]f(x,t) = \mathcal A_\lambda[\varphi]f_1(x,t) +\mathcal A_\lambda[\varphi]f_2(x,t) ,
\]
where 
\[ 
\widehat{f_1}(\xi) = \Big(1- \beta_0 \big(\frac{|\xi'|}{2c_*|\xi''|}\big)\Big)\widehat f(\xi) , \qquad   \widehat{f_2}(\xi) = \beta_0 \big(\frac{|\xi'|}{2c_*|\xi''|}\big) \widehat f(\xi).
\] 

It is easy to show that $\mathcal A_\lambda[\varphi]f_1$ maps boundedly from $L^p$ to $L^q$ for $q \ge p$, due to the rapid decay of the kernel.
\begin{lem}\label{lem:fast_dec}
Let $\Gamma(u,v)=(u,v,\phi_1,\phi_2)$ for smooth functions $\phi_1,\phi_2: [-1,1]^2 \rightarrow \R$. 
Also, let $f_1$ be defined as above. Then there is a $C_N>0$ such that
\[
\| \mathcal A_\lambda[\varphi]f_1\|_{L^q(\mathbb R^4\times I)} \le C_N\lambda^{-N} \|f\|_{L^p}
\]
for any $N\ge 1$.
\end{lem}
\begin{proof}
Let us set $\mathcal A_\lambda [\varphi]f_1=K^1_\lambda(\cdot,t)\ast f $ where 
\[
K^1_\lambda (x,t) =  \chi_I(t) \int e^{i  x\cdot \xi} \int e^{  - i  t\Gamma(u,v) \cdot \xi} \varphi(u,v)\,dudv \beta_\lambda( \xi) \big(1- \beta_0( \tfrac{|\xi'|}{2c_*|\xi''|})\big)\,d\xi.
\] 
Since $|\xi'|\ge 2c_*|\xi''|$, 
we have
\[ |\nabla_{u,v}(\Gamma(u,v)\cdot\xi)| = |\xi' + J_\Phi^\intercal
 \xi''| \gtrsim |\xi' |\gtrsim |\xi|. \] 
By repeated integration by parts, we get $|\partial_{\xi}^\alpha\int e^{-it\Gamma(u,v)\cdot\xi }\varphi(u,v) dudv|\lesssim \lambda^{-N}(1+|\xi|)^{-N}$ for any $N\ge1$ and multi-indices $\alpha$.  
This implies that $|K_\lambda^1(x,t)|\lesssim \lambda^{-N} (1+|x|)^{-N}$ for any $N\ge1$ by integration by parts. Applying Young's inequality, we get the desired estimate.
\end{proof}

Now we consider the term $\mathcal A_\lambda [\varphi]f_2$.
We may assume that $\widehat f_2$ is supported in  
\[ 
\mathbb A_\lambda=\{\xi\in \R^4:  \tfrac 12 \lambda  \le |\xi''| < 2\lambda, ~ |\xi'| \le 4 c_*|\xi''|\}.
\]
For $\mathcal A_\lambda[\varphi] f_2$, we establish the following local smoothing estimate.

\begin{thm}\label{thm-LS}
Let $\Gamma$ be nondegenerate and satisfies \eqref{G-nonzero} and \eqref{holo}.
Suppose that $\supp \widehat f \subset  {\mathbb A}_\lambda $. 
If $\frac 1 q \le \min\{ \frac 1p,\, \frac 1{3p}+\frac16,\, \frac 12(1-\frac1p) 
\}$, then
\begin{equation}\label{est-LS2}
\|\mathcal A_\lambda[\varphi] f \| _{L^q(\mathbb R^4\times I)} \le C\lambda^{\gamma} \|f \|_{L^p(\mathbb R^4)}
\end{equation}
holds for $\gamma > \frac 2p -\frac 5 q$. 
\end{thm}

\subsubsection*{Normalization of $\Gamma$}

Recall that $\Gamma_\circ(u,v) = (u,v,u^2-v^2,2uv)$. 
For a given $\epsilon > 0$, we set a class $\mathfrak R (\epsilon )$ of smooth surfaces defined by   
\[  
\mathfrak R (\epsilon ) = \Big\{ \Gamma(u,v)  : \big\| \Gamma   - \Gamma_\circ \big\|_{C^\infty([-1,1]^2)} < \epsilon    \Big\},
\]
where  $\| \Gamma\|_{C^\infty([-1,1]^2)} =   \sup_{(u,v)\in[-1,1]^2}  \sum_{|\alpha|  \ge 0 }|\partial^\alpha_{u,v}  \Gamma |$.

Under the assumption \eqref{holo},
we have
\begin{equation}\label{r2}
\begin{pmatrix} \partial_{u}^2 \phi_1 &\partial_{uv}^2 \phi_1 & \partial_{v}^2 \phi_1 \\ \partial_{u}^2 \phi_2 & \partial_{uv}^2 \phi_2 & \partial_{v}^2 \phi_2 \end{pmatrix} =   \begin{pmatrix} \partial_u^2\phi_1 & -\partial_u^2\phi_2    \\  \partial_u^2\phi_2 & \phantom{-}\partial_u^2\phi_1    \end{pmatrix} \begin{pmatrix} 1 & 0  & -1 \\ 0 & 1 & 0  \end{pmatrix}  . 
\end{equation}

For simplicity, let us set $z=(u,v)\in [-1,1]^2$.  
For $w   \in [-1,1]^2$, 
we normalize the surface $\Gamma$  by
\begin{equation}\label{normal-gamma}
\Gamma_{  w}^h(z) = (\mathrm M^h_w)^{-1} \big(\Gamma(hz+ w)  - \Gamma( w)   \big) ,
\end{equation}
where \[
\mathrm M^h_w  = 
\begin{pmatrix} 
\,\,  h & \, 0  & \,0 &  \,0 \,  \\
\,\,  0 & \, h & \, 0 &  \,0 \,\\
\multicolumn{2}{l}{h\nabla_{z}\phi_1(w)} & \,h^2 \partial_u^2\phi_1(w)  & \,-h^2\partial_u^2\phi_2(w)\, \\
\multicolumn{2}{l}{h \nabla_{z}\phi_2(w)} &\,h^2\partial_u^2\phi_2(w) & \,\phantom{-}h^2 \partial_u^2\phi_1(w) \,
\end{pmatrix} .
\]

For a given $\epsilon$, we can choose $h$ so that $\Gamma_w^h \in \mathfrak R (\epsilon)$ under the condition \eqref{G-nonzero}, as stated in the following lemma. 
  
\begin{lem}\label{lem:normal}
Let $\Gamma$ be nondegenerate and satisfies \eqref{G-nonzero}.  
For $\epsilon>0$, there exists constant $h_\circ=h_\circ(\epsilon )>0$ such that $\Gamma_{w}^h(z) \in \mathfrak R (\epsilon ) $ if $ 0< h < h_\circ$.
\end{lem}

\begin{proof}
By the Taylor expansion at $w$, 
\begin{align*}
 \phi_i (hz+w)  -  \phi_i(w) 
 &=   \nabla_{z}\phi_i(w)\cdot (hz) +   \frac {1}{2}( hz)^\intercal   \nabla_z^2 {\phi_i}(w)   (hz) + O(|hz|^3)
\end{align*}
holds for $i=1,2$. 
Using \eqref{r2}, it follows that 
\begin{align*}
&  \Gamma ( hz+w ) - \Gamma(w) \\   
&= \begin{pmatrix}  
h&0&0&0&0  \\  0&h&0&0&0 \\ 
h\partial_u \phi_1  & h\partial_v \phi_1   &  h^2\partial_{u}^2 \phi_1 & h^2\partial_{uv}^2 \phi_1 & h^2\partial_{v}^2 \phi_1 \\ 
h\partial_{u } \phi_2 & h\partial_{v } \phi_2   &  h^2\partial_{u}^2 \phi_2 & h^2\partial_{uv}^2 \phi_2 & h^2\partial_{v}^2 \phi_2   
\end{pmatrix} 
\begin{pmatrix}  u \\  v \\ u^2  \\ 2uv \\ v^2 \end{pmatrix} + O(|hz|^3)  \\
&= \begin{pmatrix}  h&0&0&0  \\  0&h&0&0 \\ 
h\partial_u \phi_1  & h\partial_v \phi_1    & h^2 \partial_u^2\phi_1   &  -h^2\partial_u^2\phi_2  \\ 
h\partial_{u } \phi_2  & h\partial_{v } \phi_2   & h^2 \partial_u^2\phi_2   &  \phantom{-}h^2\partial_u^2\phi_1     \end{pmatrix} \begin{pmatrix}  u \\  v \\ u^2 -v^2  \\ 2uv   \end{pmatrix} + O(|hz|^3) \\[1ex]
& = \mathrm M_w^h \Gamma_\circ(z)  + O(|hz|^3).
\end{align*}

Hence, we get
\[
\Gamma_w^h(z) - \Gamma_\circ(z) =  \|(\mathrm M_w^h)^{-1}\|  O(|hz|^3)  =  O(h).
\] 
Then  $\Gamma_w^h(z) \in \mathfrak R(\epsilon)$ by taking a sufficiently small $h_\circ \le \epsilon$.
\end{proof}

We choose a constant $c_\circ<1/2 $. 
Since $\varphi$ is compactly supported in $[-1,1]^2$, there exists a finite collection of balls $\{\mathbf B\}$  of radius $c_\circ/4$ that covers the support of $\varphi$.
Without loss of generality, we may assume that $\varphi$ is supported in a ball $\mathbf B$ such that $\mathbf B \subset B^2(0, c_\circ/2)$ since the operator norm is invariant under translations.

Furthermore, it is enough to consider the case of $\Gamma \in \mathfrak R(\epsilon_\circ)$.
By Lemma \ref{lem:normal}, we choose $h_\circ $ such that $\Gamma_w^h(z) \in \mathfrak R(\epsilon_\circ)$ for $0<   h  <   h_\circ$.
We decompose $\mathbf B$ into squares $\fq$ of side length $ h $.
Let $\widetilde \chi\in C_0^\infty([-1,1]^2)$ such that $\widetilde \chi \ge0$ and $\sum_{k \in \mathbb Z^2} \widetilde \chi(\cdot-k)=1$. 
We set $\widetilde \chi_{\fq} (z)=\widetilde \chi(h^{-1}( z - w_{ \fq}))$ where $ w_{ \mathfrak  q}$ is the center of the square $  \mathfrak q$, and $\varphi_{\mathfrak q}=\varphi \widetilde \chi_{\fq}$. 
Then we have
\begin{equation}\label{decomp1}
\mathcal A_\lambda[\varphi] f=\sum_{\fq} \mathcal A_\lambda[\varphi_{ \mathfrak  q}]f.
\end{equation}

By the definition of $\Gamma_{w}^h$ (see \eqref{normal-gamma}), we have
\[
x - t \Gamma(hz+w) = (x-t \Gamma(w)) - t\mathrm M_w^h\Gamma_{w}^h(z).
\]

Let us set $\mathcal A_\lambda = \mathcal A_{\lambda,\Gamma}$. 
By a change of variables $z\rightarrow hz+ w$, 
we get
\begin{align*} 
\mathcal A_{\lambda,\Gamma}[\varphi_\fq] f(x,t)  = h^2 \mathcal A_{\lambda,\Gamma_{w}^h} [\widetilde\varphi] f_h \big( (\mathrm M_w^h)^{-1}(x-t\Gamma(w)),t \big),
\end{align*}
where $\widetilde\varphi (z)  =  \varphi_{\mathfrak q} ( hz+w)$ and $f_h(x)=f(\mathrm M_w^h\,x)$.
By a change of variables  $x \to  \mathrm M_w^h x + t \Gamma(w)$, we get
\begin{align}
  \| \mathcal A_{\lambda,\Gamma} [\varphi_\fq] f \|_{L^q(\mathbb R^4\times I)}    
 & = h^2 \Big( \iint  |\mathcal A_{\lambda,\Gamma_{w}^h} [\widetilde\varphi ] f_h ((\mathrm M_w^h)^{-1}(x-t\Gamma(w)),t )|^q dx dt \Big)^{\frac1q} \nonumber \\
 & = h^2 |\det (\mathrm M_w^h)|^{\frac 1q}  \| \mathcal A_{\lambda,\Gamma_{w}^h} [\widetilde\varphi ] f_h \|_{L^q(\mathbb R^4\times I)} . \label{rescaled af}
\end{align}

Since $\|f_{h}\|_p = h^{-6/p} \|f\|_p$ for a fixed constant $h $, it suffices to prove \eqref{est-LS2}  for $\gamma > \frac 2 p -\frac 5 q$ under the assumptions  $\Gamma \in \mathfrak R (\epsilon_\circ)$, $\varphi \in C_0^\infty(\mathbf B)$, and
$\supp \widehat f \subset  \mathbb A_\lambda$ with $\|f\|_{L^p(\mathbb R^4)} =1$. 

For this purpose, we define the induction quantity. 

\begin{defn}
For $1\le p \le q \le \infty$ and
$\lambda \ge 1$, let   
\[ 
\begin{aligned}
\mathscr Q(\lambda) = \sup \big\{ \| \mathcal A_{\lambda}[\varphi]f\|_{L^q(\mathbb R^4\times I)}\, : \, 
& \Gamma \in \mathfrak R (\epsilon_\circ), \varphi \in C_0^\infty (\mathbf B),  \\
& \supp \widehat f \subset \mathbb A_\lambda, \|f\|_{L^p(\mathbb R^4)} \le 1 \big\}.
\end{aligned}
\]
\end{defn}
It is obvious that $\mathscr Q(\lambda)$ is finite since $\mathscr Q(\lambda)\lesssim \lambda^4$. 
We claim that 
\begin{equation}\label{qlambda}
\mathscr Q(\lambda)\lesssim\lambda^{\gamma} ~\text{ for }~ \gamma > 2/ p - 5/ q .  
\end{equation}

For the rest of the paper, we fix $\epsilon_\circ$ such that
\begin{align}\label{ordering}
\lambda^{-\frac12}  \ll h \le h_\circ <\epsilon_\circ
\end{align}
for the constant  $h_\circ=h_\circ(\epsilon_\circ)$  in Lemma \ref{lem:normal}.

Let $\mathfrak C(h)$ be the collection of squares $\fq$ of side length $h$. 
For each $\fq\in \mathfrak C(h)$ centered at $w_\fq$, we define
\begin{align}\label{rectangle1}
\mathbf P(\fq) = \mathbf P_{\lambda}(\fq) = \big\{\xi\in \mathbb A_\lambda: \big|\nabla_{u,v}(\Gamma(w_{\fq})\cdot\xi ) \big| \le 8  \lambda h \big\}. 
\end{align}  
By $f_{\mathbf P(\fq)}$ we denote a frequency localized function  such that 
\[  
\widehat{f_{\mathbf P(\fq)}}(\xi) =  \beta_{\lambda,\fq}(\xi)  \widehat f(\xi)
\]
for $ \beta_{\lambda,\fq} (\xi) = \beta_0 \big(  |\nabla_{u,v}\big(\Gamma(w_{\fq})\cdot\xi\big)| / ( 4\lambda h ) \big) $.

We observe that the frequency support of $\mathcal A[\varphi_{\fq}]f$ is essentially localized in $\mathbf P(\fq)$ in the sense that $\mathcal A_{\lambda}[\varphi_{\fq}](f-f_{\mathbf P(\fq)})$ is negligible.

\begin{lem}\label{lem:away_crit}
For $\epsilon_\circ, h $ satisfying \eqref{ordering}, 
let $\Gamma \in \mathfrak R (\epsilon_\circ)$, $\fq \in \mathfrak C(h)$,
and $\varphi_{\fq} \in C_0^\infty(\fq)$. 
Suppose  $\widehat f$ is supported on $ \mathbb A_\lambda$.
For any $1\le p \le q \le \infty$ and $N\ge1$,
\begin{equation}\label{est:away_crit}
\big\| \mathcal A_\lambda [\varphi_{\fq}](f-f_{\mathbf P(\fq)})\big\|_{L_{x,t}^q(\R^4 \times I)} \le C\lambda^{-N}\|f\|_p .
\end{equation}
\end{lem}

\begin{proof}
Let us set 
\[
 \mathcal A_\lambda [\varphi_{\fq}](f-f_{\mathbf P(\fq)})(x,t)=K_{\lambda, \fq}(\cdot,t)\ast f(x) ,
 \] 
 where 
\[
K_{\lambda,\fq}(x,t)=\int e^{ i x\cdot \xi} \int e^{-it \Gamma(u,v)\cdot \xi}\varphi_{\fq}(u,v)\,dudv (1-\beta_{\lambda,\fq}(\xi))\beta_\lambda(\xi ) d\xi.
\]

For each $z =(u,v)  \in\fq$ and $\xi \notin \supp \beta_{\lambda,\fq}$,  we claim that  
\[ 
|\nabla_{z}\big(\Gamma(z)\cdot \xi\big) |\ge  4\lambda h. 
\]
By the  Taylor expansion at the center $w_\fq$ of $\fq$, we have
\[
\nabla_z (\Gamma(z)\cdot\xi) = \nabla_z (\Gamma(w_\fq)\cdot\xi) + \nabla_z^2 (\Gamma(w_\fq)\cdot\xi) (z-  w_\fq)^\intercal  +O(\lambda |z- w_\fq|^2) .
\]
It follows that
\begin{align*}
| \nabla_z \big(\Gamma(z)\cdot\xi\big)| 
& \ge |\nabla_z \big( \Gamma(w_\fq)\cdot \xi \big)| - | \nabla_z^2 (\Gamma(w_\fq)\cdot \xi) (z-  w_\fq)^\intercal +O(\lambda|z- w_\fq|^2) | 
 \ge  4  \lambda h   
\end{align*}
since $\| \nabla_z^2 (\Gamma(w_\fq)\cdot\xi)\|  = \|\xi_3\nabla_z^2\phi_1(w_\fq)+ \xi_4 \nabla_z^2\phi_2(w_\fq)\| \le 3\lambda $ for $\Gamma \in \mathfrak R(\epsilon_\circ)$ and  $|z - w_\fq|  \le h $.

Repeated integration by parts yields that for any $N\ge1$ and $\alpha$,
\begin{equation*}
\big|\partial_\xi^\alpha\int e^{-it \Gamma(u,v)\cdot \xi}\varphi_{\fq}(u,v)\,dudv \big|\le C_N h^2  (1+ \lambda h^2 )^{-N}
\end{equation*}
provided by $\xi \notin \supp \beta_{\lambda,\fq}$. 
By \eqref{ordering}, 
we have $\|K_{\lambda,\fq}(\cdot,t)\|_1\lesssim \lambda^{-N/2}$. Thus the desired estimate \eqref{est:away_crit} follows by Young's convolution inequality.
\end{proof}

\subsubsection*{Linear and multilinear decomposition}

In order to show \eqref{qlambda}, we modify the argument in \cite[Lemma 2.8]{HL} (see also \cite[Lemma 4.1]{KLO23}, \cite{HKL2}). 
We decompose the averaging operator $\mathcal A_\lambda[\varphi] $ into two parts: linear operator defined on a small domain of $\Gamma$ and a  product of operators with separated subsets of domains of $\Gamma$.

\begin{lem}\label{lem-multi-decomp}
 For each $(x,t)\in \mathbb R^4\times I$, there exists a constant $C>0$ such that
\begin{align}\label{multi-decomp}
|\mathcal A_\lambda[\varphi] f(x,t) | 
&\le C \max_{\mathfrak q} |\mathcal A_\lambda[\varphi_{\mathfrak q}] f(x,t)| 
+ h^{-2}\max_{\substack{(\mathfrak q_1,\mathfrak q_2,\mathfrak q_3)\\ \in \mathfrak D(h)}}
\prod_{i=1}^3  |\mathcal A_\lambda [\varphi_{\mathfrak q_i}] f(x,t)|^{\frac13}
\end{align}
where $\mathfrak D(h)  = \big\{ (\mathfrak q_1,\mathfrak q_2,\mathfrak q_3) : \mathfrak q_i \in \mathfrak C(h),~ \min_{i\neq j}\dist(\mathfrak q_i,\mathfrak q_j) \ge h \big\}$. 
\end{lem}
\begin{proof}
Let $(x,t)$ be fixed. Also let  $\mathfrak q_1^\ast$ be a square such that  
\[
|\mathcal A_\lambda[\varphi_{\mathfrak q_1^\ast }] f(x,t)| = \max_{\mathfrak q} |\mathcal A_\lambda[\varphi_{\mathfrak q}] f(x,t)| .
\]
We define a family of squares away from $\mathfrak q_1^\ast$ by 
\[
\mathfrak D_1 = \{ \mathfrak q  : \dist (\fq, \fq_1^\ast) \ge h \}  . 
\]
Let $\mathfrak q_2^\ast\in \mathfrak D_1$ be such that $|\mathcal A_\lambda[\varphi_{\mathfrak q_2^\ast }] f(x,t)| = \max_{\mathfrak q\in \mathfrak D_1} |\mathcal A_\lambda[\varphi_{\mathfrak q}] f(x,t)|$.
In a similar way, we can define $\mathfrak D_2 $ and $\mathfrak q_3^\ast\in \mathfrak D_2$. 

Note that if $\dist (\mathfrak q, \mathfrak q_k^\ast) \ge h$ for all $k=1,2$, then 
\[
\max_{\mathfrak q} |\mathcal A_\lambda[\varphi_{\mathfrak q }] f(x,t)| \le  |\mathcal A_\lambda[\varphi_{\mathfrak q_k^\ast}] f(x,t)| ~\text{ for all }~ k=1,2,3. 
\]
Setting $\mathfrak D_3 = \cup_{k=1}^2 \{  \mathfrak q  :\dist (\mathfrak q, \mathfrak q_k^\ast) < h  \}$, we can divide a collection of  squares $\mathfrak q$ into two parts $\mathfrak q \in \mathfrak D_3$ or $\mathfrak q \notin \mathfrak D_3$.
Then we have 
\begin{align*}
|\mathcal A_\lambda[\varphi] f(x,t)| &  =\sum_{ \mathfrak q\in \mathfrak D_3 } |\mathcal A_\lambda[\varphi_{ \fq}]f(x,t) | + \sum_{ \mathfrak q\notin \mathfrak D_3 } |\mathcal A_\lambda[\varphi_{ \fq}]f(x,t)| \\
& \le C \max_{\mathfrak q} |\mathcal A_\lambda[\varphi_{ \fq}]f(x,t) | +  h^{-2}  \max_{\mathfrak q_1,\mathfrak q_2,\mathfrak q_3 \in \mathfrak D(h)} \prod_{i=1}^3  |\mathcal A_\lambda[\varphi_{ \fq_i}]f(x,t) |^{\frac13}
\end{align*}
since $q\notin \mathfrak D_3$ implies $\dist(\mathfrak q, \mathfrak q_k^\ast) \ge h$ for all $k=1,2$.
\end{proof}

We treat the linear term on the right-hand side of \eqref{multi-decomp} via rescaling.
 
\begin{lem} 
Let $\lambda,\epsilon_\circ, h$ satisfy \eqref{ordering}. Let $\Gamma 
\in\mathfrak R (\epsilon_\circ)$, $\fq \in  \mathfrak C(h)$, and  $\supp \widehat f  \subset   \mathbb A_\lambda$.
For any $N\ge1$ and $1\le p \le q \le \infty$, there are $C, C_N>0$ such that
\begin{align}\label{linear 1}
\big\|\max_{\fq}|\mathcal A_\lambda [\varphi_{\fq}]f|\big\|_{q}
\le C h^{2(1+\frac3q-\frac3p)}
\mathscr Q ( h^2 \lambda) 
\|f\|_p  +C_N\lambda^{-N}\|f\|_p .
\end{align}
\end{lem}

\begin{proof}
By embedding $\ell^q \subset \ell^\infty$, we have
\[
\big\|\max_{\fq}|\mathcal A_\lambda [\varphi_{\fq}]f|\big\|_{q}^q
\le \sum_{\fq} \big\|\mathcal A_\lambda [\varphi_{\fq}]f\big\|_q^q
\le \mathcal I_1+\mathcal {I}_2
\]
where
\begin{align*}
\mathcal I_1=\sum_{\fq}\big\|\mathcal A_\lambda    [\varphi_{\fq}]f_{\mathbf P(\fq)}\big\|_q^q, \quad
\mathcal I_2=\sum_{\fq} \big\|\mathcal A_\lambda [\varphi_{\fq}](f-f_{\mathbf P(\fq)})\|_q^q.
\end{align*}
Since the number of $\fq$ is bounded by $4h^{-2} \ll \lambda $,   Lemma \ref{lem:away_crit} yields $\mathcal I_2 \lesssim \lambda^{-N}\|f\|_p$ for $1\le p \le q \le \infty$.

For $\mathcal I_1$,  recall \eqref{rescaled af}:
\begin{align}
  \| \mathcal A_\lambda [\varphi_\fq] f \|_{L^q(\mathbb R^4\times I)}    
 & = h^2 |\det (\mathrm M_w^h)|^{\frac 1q}  \| \mathcal A_{\lambda,\Gamma_{w}^h} [\widetilde\varphi ] f_h \|_{L^q(\mathbb R^4\times I)} , \label{rescaled af-1}
\end{align}
where $\widetilde\varphi (z)  = \varphi( hz+w)$ and 
$f_h(x)=f_{\mathbf P(\fq)}(\mathrm M_w^h\,x)$.
Since the Fourier transform $\widehat{f_h}(\xi) = |\det ( \mathrm M_w^h)^{-\intercal}|  \widehat{f_{\mathbf P(\fq)}}((\mathrm M_w^h)^{-\intercal} \xi)$
is contained in $\mathbb A_{h^2\lambda}$ and $\widetilde\varphi$ is supported in $\mathbf B$,
it follows from the definition of $\mathscr Q(\lambda)$ that
\begin{align*}
\| \mathcal A_{\lambda , \Gamma_{w}^h} [\widetilde\varphi ] f_h \|_{L^q(\mathbb R^4\times I)}  
& \le  |\det (\mathrm M_w^h )|^{ -\frac1p} \mathscr Q  (h^2 \lambda) \|f\|_p .
\end{align*}
By applying this to \eqref{rescaled af-1}, we obtain
\begin{equation*} 
\| \mathcal A_\lambda[\varphi_{\fq} ] f\|_{L^q(\mathbb R^4\times I)} \le C h^{2 + \frac 6q - \frac 6p} \mathscr Q  (h^2 \lambda)  \| f_{\mathbf P(\fq)}\|_{p}
\end{equation*}
 since $|\det (\mathrm M_w^h )| = h^6$. 
By embedding $\ell^p \subset \ell^q$ ($p\le q$),  
\begin{align*}
\mathcal I_1 \lesssim h^{2q(1+\frac3q-\frac3p)}
\mathscr Q(h^2 \lambda)^q
\Big(\sum_{\fq}\|f_{\mathbf P(\fq)}\|_p^p\Big)^{\frac qp}.
\end{align*}

Therefore, to show \eqref{linear 1}, it suffices to show that
\begin{align}\label{little sum}
\big(\sum_{\fq} \big\| f_{\mathbf P(\fq)} \|_p^p \big)^{\frac1p} \le C \|f\|_p,~ \quad  2\le p \le \infty .
\end{align}
This follows by interpolation between trivial $L^\infty$ estimate and $L^2$ estimate, which can be  deduced from Plancherel's theorem and bounded overlapping property of $\supp \beta_{\lambda,\fq}$.

To verify the bounded overlapping property of $\supp \beta_{\lambda,\fq}$, suppose that $\fq$ and $\fq'$ are separated by $C h$ for a constant $ 1 \ll C <(2h)^{-1}$. 
By Taylor's theorem, we have  
\begin{align*}
&\nabla_{z} \big(\Gamma(w_{\fq}) \cdot \xi\big)
-\nabla_{z} \big(\Gamma(w_{\fq'}) \cdot \xi \big)\\
& =\big(\xi_3 \nabla_z^2\phi_1(w_{\fq'})+\xi_4 \nabla_z^2\phi_2(w_{\fq'} )\big)(w_\fq-w_{\fq'}) + O( \epsilon_\circ \lambda |w_\fq -w_{\fq'}|^2).
\end{align*}
Since $\Gamma\in \mathfrak R(\epsilon_\circ)$, it follows that $ | \xi_3 \nabla_z^2\phi_1+\xi_4 \nabla_z^2\phi_2| \gtrsim |\xi''| \gtrsim \lambda $.
Then we get 
\begin{equation*}
\big|\nabla_{z} \big(\Gamma(w_{\fq}) \cdot \xi\big)
-\nabla_{z} \big(\Gamma(w_{\fq'}) \cdot \xi \big) \big|
\ge C h \lambda -   8\epsilon_\circ \lambda (Ch)^2  
\gtrsim \lambda h.  
\end{equation*}
Thus $\xi$ cannot be contained both $\mathbf P(\fq)$ and $\mathbf P(\fq')$ simultaneously. 
\end{proof}

For the trilinear part  in \eqref{multi-decomp}, we claim the following $L^p$--$L^q$ estimates holds.

\begin{prop}\label{multi-interpol}
Let $\lambda, \epsilon_\circ, h$ be constants satisfying \eqref{ordering}.
Suppose $\Gamma \in \mathfrak R (\epsilon_\circ)$, $\supp \widehat f \subset \mathbb A_\lambda$,  and $(\mathfrak q_{1}, \mathfrak q_2, \mathfrak q_3) \in \mathfrak D(h)$.
If $\frac 1 q \le \min\{ \frac 1p,\, \frac 1{3p}+\frac16,\, \frac23(1-\frac1p)\}$, then
\begin{align}\label{tri-result}
\Big\| \prod_{j=1}^3 |\mathcal A_\lambda  [\varphi_{\mathfrak q_j}]f|^{\frac13} \Big\|_{L_{x,t}^q(\R^4\times I)} \lesssim \lambda^{ \gamma}
\|f\|_p
\end{align}
holds for $\gamma >   \frac2p-\frac5q$. 
\end{prop}

This can be achieved by making use of a trilinear restriction estimate and decoupling inequalities for associated Fourier extension operator, which will be discussed in the next two sections.

\section{A trilinear restriction estimate}\label{TRI}

In this section, we show that \eqref{tri-result} holds with $p=2$ and $q=3$ by using the multilinear restriction estimates established by Bennett--Carbery--Tao \cite{BCT}.
In order to apply the multilinear restriction estimate, we need to decompose the operator $\mathcal A_\lambda  [\varphi_{\mathfrak q}]f$ into a sum of Fourier restriction operators and an error term. 
We make use of the method of stationary phase for the multiplier of $\mathcal A_\lambda  [\varphi_{\mathfrak q }]f $, which is given by 
\[
\mathrm m_\fq (t\xi) =  \int e^{-it  \Gamma(z)\cdot \xi} \varphi_{\fq}(z )\,dz. 
\]
Using the assumption \eqref{G-nonzero} on $\Gamma$ and the implicit function theorem, we will find the unique stationary point $z = z(\xi)$ and associated hypersurface $\Psi_\fq(\xi) = \Gamma(z(\xi))\cdot \xi$ defined in a small neighborhood of the center of $\mathbf P(\fq)$. 

Let us define the Fourier extension operator $T_\fq$ associated with $(\xi,\Psi_\fq(\xi))$ by  
\[  
T_\fq [ \psi] g(x,t)=\int  e^{i  ( x\cdot \xi - t\Psi_\fq (\xi)  ) } \psi( \xi)  g(\xi)\,d\xi ,
\]
where $\psi$ is a compactly supported smooth function in the domain of $g$.

\begin{lem}\label{average-sum} 
For a sufficiently large integer $N$, 
\[ 
\mathcal A_\lambda[\varphi_{\fq}] f_{\mathbf P(\fq)}(x,t)   = \chi_I(t) \sum_{l=0}^{N+5} t^{-1-l}   T_\fq [\mathbf b_{\fq,l}|\xi|^{-1-l}] \widehat {f_{\mathbf P(\fq)}}   (x,t)+ \mathcal E f   (x,t),
\]
 where  $\mathbf b_{\fq,l} (\xi)    $ is a smooth function such that $|\partial_\xi^\alpha \mathbf b_{\fq,l} | \lesssim |\xi|^{-|\alpha|}$, and $\|  \mathcal E  f  \|_q \le \lambda^{-N}\|f\|_p$ for any $1 \le p \le q \le \infty$. 
\end{lem}

 \begin{proof}
For the center $w_\fq$ of $\fq$, we have 
\begin{align}\label{e-m}	
\mathrm m_\fq (t \xi) =  e^{ -i t \Gamma(w_\fq)\cdot \xi }\int e^{-i (\Gamma(z) - \Gamma(w_\fq))\cdot \xi } \varphi_\fq(z) dz .
\end{align}
We consider a smooth real valued  function 
\[  
\Psi_\fq (z, \xi) := (\Gamma(z)  - \Gamma(w_\fq) )\cdot \xi. 
\]
For each $\xi'' =(\xi_3,\xi_4)$, define a point 
$\xi'_\fq :=-\mathrm J_\Phi^\intercal(w_\fq) \xi''$ where
\begin{align}\label{Jwq}
\mathrm J_\Phi^\intercal(w_\fq) = \begin{pmatrix}
  \partial_u \phi_1(w_\fq) & \partial_u \phi_2(w_\fq)  \\ \partial_v \phi_1(w_\fq) & \partial_v \phi_2(w_\fq) \end{pmatrix}.
\end{align}

Then 
$
\Psi_\fq\big(w_\fq, (\xi'_\fq,\xi'')\big) = \nabla_z \Psi_\fq\big(w_\fq ,(\xi'_\fq, \xi'')  \big)= 0 
$
and 
\[
\det \nabla^2_{z} \Psi_{\fq} (w_\fq,( \xi'_\fq, \xi''))  = \det(\xi_3 \nabla^2_{z} \phi_1(w_\fq) + \xi_4 \nabla^2_{z} \phi_2(w_\fq) ) \neq 0  
\]
since  $\xi''\neq 0$ by the assumption $\xi \in \mathbb A_\lambda$.
By the implicit function theorem, there exist a neighborhood $V_\fq\subset \mathbb R^4$ of $( \xi'_\fq, \xi'')$ and a unique function $z : V_\fq \to \mathbb R^2$ such that  
\begin{align}\label{nablazero}
 z( \xi'_\fq, \xi'') = w_\fq ~\text{ and }~  \nabla_{z} \Psi_\fq \big(z(\xi),\xi \big)= 0  
\end{align}
for all $\xi \in V_\fq$. 
Note that $( \xi'_\fq, \xi'')$ is the center of $\mathbf P(\fq)$.
The method of stationary phase  (e.g., \cite[Theorem 7.7.6]{H}) applied to \eqref{e-m} gives
\begin{align}\label{expand_m}
\mathrm m_\fq (t \xi)
& = e^{ -i t \Gamma(w_q))\cdot \xi }
e^{it(\Gamma(z(\xi)-\Gamma(w_q))\cdot \xi}\Big(\sum_{l=0}^k \mathbf b_{\fq,l}(\xi) (t |\xi| )^{-1-l} + \mathbf e_{\fq,k} (\xi) \Big),
\end{align}
where $\mathbf b_{\fq,l} (\xi)    $ and $\mathbf e_{\fq,k}  $ are smooth functions such that $|\partial_\xi^\alpha \mathbf b_{\fq,l} | \lesssim |\xi|^{-|\alpha|}$ and $|\partial_\xi^\alpha \mathbf e_{\fq,k} (\xi)| \lesssim |\xi|^{-k-|\alpha| }$ for every multi-index $\alpha$. 
We set
\begin{equation}\label{phase-multiplier}
\Psi_\fq (\xi ) := \Gamma(z(\xi))\cdot \xi .
\end{equation}

Applying \eqref{expand_m}, it follows that 
\begin{align*}
&\mathcal A_\lambda[\varphi_{\fq}] f_{\mathbf P(\fq)}(x,t)\\
&=\chi_I(t)  \int e^{i x\cdot \xi} \mathrm m_\fq(t\xi)   \widehat{f_{\mathbf P(\fq)}}(\xi)  d\xi  \\
&=\chi_I(t) \int e^{i (x\cdot \xi - t   \Psi_\fq (\xi) ) }  \Big( \sum_{l=0}^k \mathbf b_{\fq,l}(\xi) (t |\xi|)^{-1-l}    \Big)   \widehat{f_{\mathbf P(\fq)}}(\xi)  d\xi + \mathcal Ef(x,t)
\end{align*}
where
\[
  \mathcal E f (x,t) = \chi_I(t) \int e^{i (x\cdot \xi - t   \Psi_\fq(\xi) ) }    \mathbf e_{\fq,k} (\xi)     \beta_{\lambda,\fq}(\xi) \widehat{f}(\xi)  d\xi  .
 \]
By the decay rate of $|\partial^{\alpha}_\xi\mathbf e_{\fq,k}(\xi)|\lesssim |\xi|^{-k-|\alpha|}$, integration by parts gives that $\|  \mathcal E  f  \|_q \le \lambda^{-N}\| f \|_p$ for any $1 \le p \le q \le \infty$ by taking $k >N+5$.
\end{proof}

We observe some properties of the function $ \Psi_\fq(\xi)$ given in \eqref{phase-multiplier}. 

\begin{lem}\label{rank-hess-2}
Let $\Psi_\fq$ be the phase function defined in \eqref{phase-multiplier}. Then 
\begin{enumerate} 
\item[$($i$\,)$]   $\Psi$ is homogeneous of degree 1, i.e., $\Psi_\fq(s\xi)= s\Psi_\fq(\xi)$ for all $s \neq0$, $\xi\in \R^4$.\vspace{.1cm}
\item[$($ii$\,)$] 
For any $\xi\in \R^4\setminus\{0\}$ such that $|\xi''|\neq 0$, $\nabla_{\xi'}\Psi_\fq (\xi)$ is nonsingular. Consequently, \[\rank  \nabla_{\xi}^2 \Psi_\fq (  \xi) \ge 2.\]
\end{enumerate}
\end{lem}

\begin{proof}
$(i)$ Since $\Psi_\fq (\xi) = \Gamma(z(\xi))\cdot \xi$, it suffices to show that $z(\xi)$ is of homogeneous of degree of $0$. 
Note that
\begin{align}\label{deriv-0}
0=\nabla_z \Psi_{\fq}(z(\xi),\xi) 
=\xi' + \xi_3 \nabla_z \phi_1(z(\xi)) + \xi_4 \nabla_z \phi_2(z(\xi)).
\end{align}
Then, we also have  $0=\nabla_z\Psi_{\mathfrak q}(z(s\xi),s\xi) = s\nabla_z\Psi_{\mathfrak q}(z(s\xi),\xi)$ for any $s\neq0$, which implies
 both $z(s\xi)$ and $z(\xi)$ are the solution of the equation \eqref{deriv-0}. 
Thus we have $z(s\xi)=z(\xi)$, as desired. 

$(ii)$ By the Leibniz rule and \eqref{nablazero}, we have
\begin{align}\label{grad1}
\nabla_{\xi}\Psi_\fq (\xi)=\nabla_\xi ( \Gamma(z(\xi) \cdot \xi )= \Gamma(z(\xi)). 
\end{align}
Let $z(\xi) = (u(\xi),v(\xi))$. Taking derivatives relative to $\xi'$, we obtain
\[
\nabla_{\xi'}^2 \Psi_\fq = \begin{pmatrix} 	\partial_{\xi_1} u(\xi) & \partial_{\xi_2} u(\xi) \\
	\partial_{\xi_1} v(\xi) & \partial_{\xi_2} v(\xi) 
 \end{pmatrix} .
\]

The derivatives with respect to $\xi_1$, $\xi_2$ of \eqref{deriv-0} yield
\begin{align*} 
\begin{pmatrix} 1 & 0  \\ 0 & 1
\end{pmatrix} +
\left(\xi_3 \nabla_{z}^2 \phi_1 + \xi_4  \nabla_{z}^2 \phi_2 \right) \begin{pmatrix} 	\partial_{\xi_1} u(\xi) & \partial_{\xi_2} u(\xi) \\
	\partial_{\xi_1} v(\xi) & \partial_{\xi_2} v(\xi) 
 \end{pmatrix} = 0.
\end{align*}
Hence $\nabla_{\xi'}^2 \Psi_\fq$ is non-singular by the assumption \eqref{G-nonzero}. \end{proof}

For example,  
if $\Gamma = \Gamma_\circ$ ($\phi_1 = u^2-v^2$ and $\phi_2=2uv$), a straightforward computation yields
\[
z(\xi ) =-\frac{1}{2|\xi''|^2 }\big(   \xi_1 \xi_3 + \xi_2 \xi_4 ,   \xi_1 \xi_4 - \xi_2 \xi_3   \big), ~
\Psi_\fq(\xi)=-\frac{1}{2|\xi''|^2 }\big( (\phi_1(\xi'),\phi_2(\xi')) \cdot \xi'' \big).
\] 
In this case, the surface $(\xi,\Psi_\fq(\xi))$ can be considered heuristically as a  two-dimensional conic extension of the complex curve $\Gamma_\circ=(z, z^2)$. 
If we assume $\xi'=z$ and $\xi''=w$ are complex numbers, then 
\[
\big( \xi,  \Psi_\fq(\xi) \big) = \big(z,w, \mathrm{Re}( z^2/w) \big),
\]
where $\mathrm{Re}(z)$ is the real part of $z$. 
In this point of view, we will show transversality of normal vectors to $\Psi_\fq$'s (see Lemma \ref{transv}).

Now, let 
\[
\mathcal S_\fq = \{(\xi, \Psi_\fq(\xi)): \xi \in \mathbf P(\fq) \} .
\]
By Lemma \ref{rank-hess-2},
the hypersurface $\mathcal S_\fq $ has at least $2$ nonvanishing  principal curvatures.
Although this geometric property allows us to utilize only low level of multilinearity, the trilinear restriction estimates are sufficient for our purpose.

Let $\mathbf n_\fq $ be the normal vector to $\mathcal S_\fq$ and let us denote by  $\textsf{vol}\,(\mathbf n_{\fq_1},\mathbf n_{\fq_2},\mathbf n_{\fq_3})$  the volume of the parallelepiped generated by $\mathbf n_{\fq_1},\mathbf n_{\fq_2},\mathbf n_{\fq_3}$.
The following  is a multilinear restriction estimate (with lower level of multilinearity) established by Bennett--Carbery--Tao \cite[Section 5]{BCT}.

\begin{thm}\label{thm:multi}  
For $j=1,2,3$, let $\mathbf n_j=\mathbf n_{\fq_j}$ be the normal vector to $\mathcal S_{\fq_j}$.  
Suppose that  
\begin{align}\label{lin-indep}
\textsf{vol}\,(\mathbf n_1 ,\mathbf n_2  , \mathbf n_3 )\ge  C \delta^3.
\end{align}
Let $d\sigma_j$ be the surface measure on $\mathcal S_j$. 
If $\lambda \gg \delta^{-2}$, there is a positive constant $C_\epsilon=C_\epsilon(\delta)$ such that
\[
\Big\|\prod_{j=1}^3 |\widehat{g_j d\sigma_j}  |^{\frac13} \Big\|_{L_{x,t}^3(B^5(0, \lambda))}
\le C_\epsilon \lambda^{\epsilon}\prod_{j=1}^3\|g_j\|_{L^2(d\sigma_j)}^{\frac13}  
\]
for $g_j \in L^2(d\sigma_j) $. 
\end{thm}

We omit the proof, while we verify the transversality condition \eqref{lin-indep} holds if $\fq_j$ are separated. Here, we need the condition that $(\phi_1,\phi_2)$ satisfies the Cauchy-Riemann equation \eqref{holo}.

\begin{lem}\label{transv}
Let $\Gamma \in \mathfrak R(\epsilon_\circ)$. 
Let $\fq_j\in \mathfrak C(h)$ be squares of side length $h$, and let $\mathbf n_j$ are normal vectors to $\mathcal S_{\fq_j} $. 
Suppose that
\begin{align}\label{supp-sep}
\min_{1 \le j <k \le 3} \dist(\fq_j,\fq_k)\ge 3 h  .
\end{align}
If $\epsilon_\circ$ is sufficiently small, then there exists $C > 0$ such that \eqref{lin-indep}
holds with $\delta = h$.
\end{lem}

\begin{proof}
By Lemma \ref{average-sum}, for each $\fq_j$, we have $z_j(\xi)= (u_j(\xi), v_j(\xi))$ for $\xi \in \mathbf P(\fq_j)$ and the normal vector
\[
\mathbf n_j (\xi)   =\big( \Gamma( z_j(\xi)), -1 \big) = \big( u_j(\xi), \, v_j(\xi), \, \phi_1(z_j(\xi)),\, \phi_2( z_j(\xi)), \,-1, \,0 \big)   . 
\]
The conjugate of $\mathbf n_j(\xi)$ is given by 
\[
\overline{\mathbf n}_j (\xi) = \big( -v_j (\xi), \,u_j(\xi), \, -\phi_2(z_j(\xi)),\,  \phi_1( z_j(\xi)), \,0, \,-1 \big) 
\]

To show \eqref{lin-indep}, it is enough to show that
\begin{align}\label{determinant} 
\det \begin{pmatrix} \mathbf n_1 & \overline{\mathbf n}_1 & \mathbf n_2 & \overline{\mathbf n}_2  & \mathbf n_3 & \overline{\mathbf n}_3\end{pmatrix}   \gtrsim |z_2 - z_1 |^2 |z_3-z_1|^2 |z_3-z_2|^2   . 
\end{align}

For  complex numbers $\omega_j (\xi) = u_j(\xi) + i v_j(\xi)$ and $\Phi (\omega_j) = \phi_1(z_j(\xi)) + i \phi_2(z_j(\xi))$,  
\eqref{determinant} is equivalent to 
\begin{align*} 
\det \begin{pmatrix} \omega_1 & \omega_2 & \omega_3  \\ \Phi(\omega_1) &\Phi(\omega_2) & \Phi(\omega_3) \\ -1 & -1 & -1 \end{pmatrix} \gtrsim |\omega_2 -\omega_1| |\omega_3 -\omega_1||\omega_3 -\omega_2| .
\end{align*}

By the generalized mean value theorem (\cite[part V, Problem 95]{PS}), we get
\begin{equation}\label{gmv}
\left|  \det \begin{pmatrix} \omega_1 & \omega_2 & \omega_3  \\ \Phi(\omega_1) &\Phi(\omega_2) & \Phi(\omega_3) \\ -1 & -1 & -1 \end{pmatrix} \right|  \gtrsim |\Phi''(\omega_\ast)| \prod_{1\le j< k \le 3} |\omega_j - \omega_k|,
\end{equation}
where $\omega_\ast$ is a point in a interior of convex hull of $\omega_1,\omega_2$, and $\omega_3$. 
Since $\Gamma \in \mathfrak R(\epsilon_\circ)$, we have $\Phi(\omega_j) = \big(u_j^2(\xi) - v_j^2(\xi)\big) + i \big(2 u_j(\xi) v_j(\xi)\big) + O(\epsilon_\circ)$ so that $|\Phi''(\omega_j)| = 1 + O(\epsilon_\circ)$. 
Note that $|\omega_j - \omega_k| = |z_j(\xi) - z_k(\xi)| $. 
By \eqref{supp-sep} and taking sufficiently small $\epsilon_\circ \ll 1$, the right-hand side of \eqref{gmv} is bounded below by $C_0 h^3$ for a positive constant $C_0$. 
Hence we see that $\mathbf n_1, \overline{\mathbf n}_1, \mathbf n_2,\overline{\mathbf n}_2 , \mathbf n_3,\overline{\mathbf n}_3$ are linearly independent. 
This implies \eqref{lin-indep} by \eqref{supp-sep} with $\delta =h$.
\end{proof}

Combining Lemma \ref{average-sum}, Theorem \ref{thm:multi}, and Lemma \ref{transv}, we obtain a local version of \eqref{tri-result}. 
\begin{lem}\label{tri-L2}
Let $\lambda\ge1$ and $\epsilon_\circ\ll h \ll 1$.
Suppose $\Gamma \in \mathfrak R (\epsilon_\circ)$, $\supp \widehat f_j \subset \mathbb A_\lambda$,  and $(\mathfrak q_{1}, \mathfrak q_2, \mathfrak q_3) \in  \mathfrak D(h)$.
There is a constant $C>0$ such that for $\gamma >-2/3$,
\begin{align}\label{centered}
\Big\|\prod_{j=1}^3\big|\mathcal A_\lambda [\varphi_{\fq_j}]f_j \big|^{\frac13} \Big\|_{L_{x,t}^3(B^5(0,10))}
\le C  \lambda^{\gamma}\prod_{j=1}^3\|f_j \|_2^{\frac13}.
\end{align}
\end{lem}
\begin{proof}
Let us set 
 \[
\mathcal A_\lambda [\varphi_{\fq_j}]f_j=\mathcal A_\lambda [\varphi_{\fq_j}]f_{\mathbf P(\fq_j)}+\mathcal A_\lambda [\varphi_{\fq_j}](f_j-f_{\mathbf P(\fq_j)}).
\]
Then 
\[
 \prod_{j=1}^3 \big| \mathcal A_\lambda [\varphi_{\fq_j}]f_j \big|^{\frac13}
\le  \prod_{j=1}^3 \big| \mathcal A_\lambda [\varphi_{\fq_j}]f_{\mathbf P(\fq_j)}\big|^{\frac13} + \sum_{\substack{\widetilde {f_j} = f-f_{\mathbf P(\fq_j)} \\ \text{for some}~j }}  \prod_{j=1}^3 \big| \mathcal A_\lambda [\varphi_{\fq_j}]\widetilde{f_j}\big|^{\frac13}  =: \mathrm{I} +\mathrm{I\!I} .
\]
Note that $\widetilde{f_j}$'s on the right-hand side are either $f_{\mathbf P(\fq_j)}$ or $(f_j-f_{\mathbf P(\fq_j)} )$, and at least one of $\widetilde{f_j}$ is $(f_j-f_{\mathbf P(\fq_j)} )$.  
By \eqref{est:away_crit} and the trivial estimate $\|\mathcal A_\lambda [\varphi_{\fq_j}]f\|_q\lesssim \lambda^C \|f\|_p$ for $1\le p \le q \le \infty$, followed by  H\"older's inequality, we have 
\[
 \|\, \textrm{I\!I} \,\|_{L_{x,t}^3(B^5(0,10))}  \le C_N \lambda^{-N/3+2C/3}\prod_{j=1}^3\|f_j\|_2^{\frac13}
\]
for any $N \ge 1$.  

Then it is enough to show $ \|\, \textrm{I}  \,\|_{L_{x,t}^3(B^5(0,10))}$  is bounded by the right-hand side of \eqref{centered}. 
By Lemma \ref{average-sum}, we have 
\begin{align*}
\mathrm{I} \le  \chi_I(t) \sum_{l_1,l_2,l_3=0}^{N+5} t^{-1-\frac{l_1+l_2+l_3}3} \prod_{j=1}^3  \big| T_{\fq_j} \big[\mathbf b_{\fq_j,l_j}|\xi|^{-1-l_j}\big] \widehat {f_{\mathbf P(\fq_j)}}  (x,t) \big|^{\frac13} + \widetilde {\mathcal E} f 
\end{align*}
where $ \| \widetilde {\mathcal E} f \|_q \lesssim \lambda^{- C N } \|f\|_p $ for any $1 \le p \le q \le \infty$.
With a change of variables $(x,t)\rightarrow \lambda^{-1} (x,t)$ and $\xi\rightarrow \lambda \xi$, Theorem \ref{thm:multi} implies that
\begin{equation*}
\begin{aligned} 
\Big\| \Big|\prod_{j=1}^3   & T_{\fq_j} \big[\mathbf b_{\fq_j,l_j}|\xi|^{-1-l_j}\big] \widehat {f_{\mathbf P(\fq_j)}}\Big|^{\frac13} \Big\|_{L_{x,t}^3(B^5(0,10))} \\
&=\lambda^{-\frac 53}
\Big\| \prod_{j=1}^3   \Big| T_{\fq_j} \big[\mathbf b_{\fq_j,l_j}|\lambda\xi|^{-1-l_j}\big] \widehat {f_{\mathbf P(\fq_j)}}(\lambda \cdot) \lambda^4 \Big|^{\frac 13}\Big\|_{L_{x,t}^3(B^5(0,10\lambda))} \\
& \lesssim  \lambda^{-\frac 53+\varepsilon} 
 \prod_{j=1}^3 \big(\lambda^{-1-l_j}\big)^{\frac13}\big\| \mathcal F(f_{\mathbf P(\fq_j)}(\lambda^{-1} \cdot)) \big\|_{L^2}^{\frac13}\\
& \lesssim \lambda^{ -\frac53-1+\frac42 +\varepsilon} \prod_{j=1}^3 \big(\lambda^{-l_j}  \big\| f_j \big\|_{L^2}\big)^{\frac13} .
\end{aligned}
\end{equation*}

The last inequality holds by reversing the change of variables and Plancherel's theorem. 
Therefore, we obtain 
\[
 \| \, \mathrm{I} \,\|_{L^3(B^5(0,10))}  \lesssim \big( \lambda^{-\frac23+\varepsilon} +\lambda^{-CN} \big) \prod_{j=1}^3\|f_j\|_2^{\frac13} 
\]
for a sufficiently large $N$.
This gives the desired estimate \eqref{centered}. 
 \end{proof}
 
 Extension of \eqref{centered} to a global one will be provided in Section \ref{sec:completion} after interpolation with a decoupling inequality.

 \section{Decoupling inequality}\label{DEC}
For each $\fq \in \mathfrak C(h)$, 
we establish an $\ell^p$-decoupling inequality for the operator $T_\fq[\psi]  g$ under the assumption
\[
\supp   g \subset \mathbb A_1 
\]
(see Theorem \ref{extension decoupling}). 
Furthermore, we may assume that $   g$ is supported in  $\mathbf P_1(\fq)$ 
(see \eqref{rectangle1} for the definition of $\mathbf P_1(\fq)$).

For $\sigma \le h$, we consider $\fm \in \mathfrak C(\sigma)$ centered at $w_\fm$ such that $\fm \subset \fq$. 
Using the notation in \eqref{Jwq}, we set 
\begin{equation}\label{set of plank}
\mathbf p(\fm) :=  \mathbf P_1(\fm) = \big\{(\xi',\xi'') \in \mathbb A_1: 
\big|\nabla_{u,v}(\Gamma(w_\fm)\cdot \xi ) \big| \le 8 \sigma \big\} .
\end{equation}
Note that $\mathbf p(\fm)$ is a parallelepiped with dimensions $\sim\sigma \times \sigma \times 1 \times 1$. 
We denote by $\mathcal P(\sigma)$ the collection of all  $\mathbf p(\fm)$ for $\fm \in \mathfrak C(\sigma)$.

Then we can partition $\mathbf P_1(\fq)$ into the pieces $\mathbf p$, where $\mathbf p \in \mathcal P(\sigma)$,  so that
\[ 
g = \sum_{\mathbf p \in \mathcal P(\sigma)} g_{\mathbf p}
\] 
where each  $ g_{\mathbf p} $ is supported in $\mathbf p$ .

\begin{thm}\label{extension decoupling}
Let  $\lambda^{-\frac12} \le h \le  1 $. 
Suppose that $\supp   g \subset \mathbf P_1(\fq)$ for $\fq \in \mathfrak C(h)$. 
Given $p\ge4$ and $\epsilon>0$, there is a $C_\epsilon>0$ such that
\begin{align}\label{extension decoupling norm}
\big\|    T_{\fq}[\psi] g \big\|_{L^p(B_\lambda)} \le C_\epsilon \lambda^{1-\frac2p+\epsilon}\Big( \sum_{\mathbf p \in \mathcal P(\lambda^{-1/2})} \big\|  T_{\fq}[\psi] g_{\mathbf p } \big\|_{L^p(W_{B_\lambda})}^p \Big)^{\frac 1p} ,
\end{align}
where $ {g_{\mathbf p}}$ is supported in $\mathbf p$ and $W_{B_\lambda}(y) = (1+ \lambda^{-1}|y-c_B|)^{-100}$ for a ball $B_\lambda \subset \mathbb R^5$ centered at $c_B$ and radius $\lambda$. 
\end{thm}

By Lemma \ref{rank-hess-2} ($ii$), the Hessian of the surfaces parametrized by $\Psi_\fq$ has rank $2$. 
We may view $\mathcal S_\fq = \{(\xi, \Psi_\fq(\xi)): \xi \in \mathbf P_1(\fq)\}$ as a two-dimensional conical extension of a nondegenerate surface in $\R^3$.
We derive the corresponding decoupling inequality from that for surfaces with nonvanishing Gaussian curvature, obtained by Bourgain--Demeter \cite{BD0, BD}.  

By the same argument as in Lemma \ref{tri-L2}, the following is a consequence of Theorem \ref{extension decoupling} and Lemma \ref{average-sum}. 
\begin{cor}\label{cor:dec4} 
Let  $\lambda \ge1$. Suppose that $\supp\widehat f \subset \mathbb A_\lambda$. For $ p \ge 4 $ and $\gamma > -3/p$, there is a constant $C_\gamma >0$ such that
\begin{equation}\label{dec4}
\|  \mathcal A_\lambda [\varphi_{\fq }] f \|_{L^p(B^5(0,10))}
\le C_\gamma \,\lambda^\gamma  \|f \|_p.
\end{equation}
\end{cor}

We present the proof of Corollary \ref{cor:dec4} at the end of this section, followed by a proof of Theorem \ref{extension decoupling}.\footnote{Note that the Cauchy-Riemann condition \eqref{holo} is not required in the proofs of Theorem \ref{extension decoupling} and Corollary \ref{cor:dec4}.}

It is well known that \eqref{extension decoupling norm} is equivalent to the decoupling inequality for functions whose Fourier transform is supported in a $\lambda^{-1}$-neighborhood of the surface 
$\mathcal S_\fq$
(for example, see \cite[Section 5]{BD2}).
We define $\delta$-neighborhood of $\mathcal S_\fq$ by
\[
\mathcal N_{\delta} (\mathcal S_\fq) = \big\{ (\xi,\Psi_{\fq}(\xi)+c) ~ : ~ \xi \in \mathbf P_1(\fq),~ |c|\le \delta \big\}.
\]
We cover $\mathcal N_{\delta} (\mathcal S_\fq)$ by caps $\theta = \theta(\fm)$ defined by 
\[
\theta(\fm) = \big\{ (\xi,\Psi_{\fq}(\xi)+c): ~\xi \in \mathbf p(\fm),~\fm \in \mathfrak C(\sigma),~ |c|\le \sigma^2 \big\},
\]
where $\fm$ satisfies $\fm \subset \fq$.
Note that each $\theta(\fm)$ has dimensions $\sim\sigma\times\sigma\times 1\times1\times\sigma^2$.
We denote by $\Theta (\sigma)$ the collection of all such caps $\theta(\fm)$, where $\fm \in \mathfrak C(\sigma)$.

If $\supp \widehat G \subset \mathcal N_{\lambda^{-1}}(\mathcal S_\fq)$, then writing $\widehat{G_{\theta(\fm)}} = \widehat{G} \chi_{\theta(\fm)}$, we have
\[
G = \sum_{\theta(\fm)\in \Theta(\lambda^{-1/2})} G_{\theta(\fm)} .
\]

Theorem \ref{extension decoupling} is equivalent to the following statement.

\begin{prop}\label{conic-ext2}
For $\epsilon>0$, there is a $C_\epsilon>0$ such that for $p\ge4$,
\begin{align}\label{prop_cone_result}
\big\|G \big\|_{L^p(\mathbb R^5)} \le C_\epsilon \lambda^{1-\frac2p+\epsilon} \Big( \sum_{\theta(\fm)} \| G_{\theta(\fm)}\|_{L^p(\mathbb R^5)}^p \Big)^{1/p}
\end{align}
whenever $\widehat G_{\theta(\fm)}$ is supported on $\theta(\fm)\in {\Theta}(\lambda^{-1/2})$.
\end{prop}

To prove Proposition \ref{conic-ext2}, we first examine  the geometric properties of   $\theta(\fm)$. 

Fix a dyadic number $K$ satisfying $K^{-1} \le h$, 
we consider $\sigma$ such that
\begin{align}\label{lambda-K-sigma}
 \lambda^{-\frac12}\le K^{ \frac12}\lambda^{-\frac12} \le \sigma  \le h . 
\end{align} 

Let $\fm \in \mathfrak C(\sigma)$. For $\eta_\circ \in \R^2$, we define a truncation of $\mathbf p(\fm)$ along the conical direction by 
\[
\mathbf p^{\eta_\circ} (\fm; K^{-1}) := \mathbf p(\fm) \cap \big\{(\xi',\eta)  : 
  |\eta - \eta_\circ| \le K^{-1}\big\}  . \]
Then the corresponding surface is given by
\[
\mathcal S_{\fq}^{\eta_\circ} (\fm;K^{-1}):=\Big\{ (\xi',\eta, \Psi_{\fq}(\xi',\eta)) \in \R^5:~ (\xi',\eta) \in \mathbf p^{\eta_\circ}(\fm;K^{-1})\Big\}.
\]
We prove Proposition \ref{conic-ext2} by applying the following lemma iteratively.

\begin{lem}\label{decomp:step}
Let $\lambda, K,\sigma$ be given in \eqref{lambda-K-sigma}.
Suppose $\widehat {G_{\theta(\fm)}}$ is supported on a $K^{-1}\sigma^2$-neighborhood of $\mathcal S_{\fq}^{\eta_\circ}(\fm;K^{-1})$ for $\fm \in \mathfrak C(\sigma)$.  For $\epsilon>0$, there is a $C_\epsilon>0$ such that for $p\ge4$,
\begin{align}\label{dcp:small scale}
\big\| G_{\theta(\fm)} \big\|_{L^p(\mathbb R^5)} \le C_\epsilon K^{1-\frac2p+\epsilon} \Big( \sum_{\theta(\fmm)} \| G_{\theta(\fmm)}\|_{L^p(\mathbb R^5)}^p \Big)^{1/p}
\end{align}
where $G_{\theta(\fm)} = \sum_{\theta(\fmm)} G_{\theta(\fmm)}$
such that $\widehat{G_{\theta(\fmm)}}$ is supported on $\theta(\fmm)$ for 
$\fmm \in \mathfrak C(K^{-\frac12}\sigma)$.
\end{lem}

To prove Lemma \ref{decomp:step}, we apply the well-known decoupling inequality due to Bourgain--Demeter \cite{BD}.
For the hypersurface $\mathcal S$ in $\R^3$ with nonvanishing Gaussian curvature,
we partition $\mathcal N_{\lambda^{-1}}(\mathcal S)$ into rectangular blocks $\widetilde{\theta}$ such that
\[
\mathcal N_{\lambda^{-1}}(\mathcal S)
=\bigcup \widetilde{\theta}
\]
where $\widetilde{\theta}$ is of dimension $\sim \lambda^{-\frac12}\times \lambda^{-\frac12}\times \lambda^{-1} $.

\begin{thm}[Bourgain--Demeter \cite{BD}]\label{dcp-nonv}
For $\epsilon>0$ and $p\ge4$, there is a constant $C_\epsilon>0$ such that
\[
\big\| \sum_{\widetilde{\theta}} F_{\widetilde{\theta}} \big\|_{L^p(\R^3)} \le C_\epsilon \lambda^{1-\frac2p+\epsilon} \Big( \sum_{\widetilde{\theta}} \|F_{\widetilde{\theta}}\|_{L^p(\R^3)}^p \Big)^{\frac 1p}
\]
where $\widehat F_{\widetilde{\theta}}$ is supported on $\widetilde\theta$, and the collection $\{\widetilde{\theta}\}$ forms a partition of  $\mathcal N_{\lambda^{-1}}(\mathcal S)$.
\end{thm}

\begin{proof}[Proof of Lemma \ref{decomp:step}]

For a fixed $\eta_\circ$, suppose the Fourier support of $G_{\theta(\fm)}$ is contained in a $K^{-1}\sigma^2$-neighborhood of $\mathcal S_{\fq'}^{\eta_\circ}(\fm;K^{-1})$, defined on $\mathbf p^{\eta_\circ}(\fm;K^{-1})$.

By change of variables and Taylor's theorem, it suffices to consider the case where $\mathbf p(\fm)$ is parallel to the coordinate axes and $\xi'$-side  is stretched to $\sim1$:
\[
\overline{\mathbf p}^{\eta_\circ} (K^{-1}) :=  
\{(\xi',\xi'')\in \mathbb A_1:~|\xi'| \le 8 ,~  |\xi'' - \eta_\circ|\le K^{-1} \big\} .
\]
Indeed, we consider the general case of 
\[
\mathbf p(\fm) = 
\big\{(\xi',\xi'') \in \mathbb A_1: 
\big| \xi'+\mathrm J_\Phi^\intercal(w_{\fm})  \xi'' \big| \le 8 \sigma \big\}
\]
since $\nabla_{u,v}(\Gamma(w_\fm)\cdot \xi ) =\xi'+ \mathrm J_\Phi^\intercal(w_{\fm})  \xi''$ (see \eqref{Jwq} for the definition of $\mathrm J_\Phi^\intercal(w_\fm)$).
We set the center of $\mathbf p(\fm)$ by 
\[
\Xi := (\xi'_\fm,\xi'')=(-\mathrm J_\Phi^\intercal(w_\fm)  \xi'', \xi'') .
\]  
By Taylor's theorem with respect to $\xi'$,  we observe that
\begin{equation}
\begin{aligned}\label{taylor1}
\Psi_{\fq} \big(\sigma\xi'+\xi'_\fm,\xi''\big)
=
&\Psi_\fq(\Xi)
+\nabla_{\xi'} \Psi_\fq(\Xi)\cdot (\sigma\xi')\\
&+\frac12(\sigma\xi')^\intercal \nabla_{\xi'}^2 \Psi_\fq(\Xi) (\sigma\xi')
+\mathcal E(\sigma\xi',\xi'') 
\end{aligned}
\end{equation}
where the error term
$\mathcal E(\xi) = O(|\xi'|^3)$.

By \eqref{nablazero} and \eqref{phase-multiplier}, we have  
$
\Psi_\fq(\Xi) = \Gamma(w_\fm)\cdot (\xi_{\fm}',\xi'')$.
By \eqref{grad1}, we also have
$
\nabla_{\xi'} \Psi_\fq(\Xi) =z(\Xi) = w_\fm.
$
Then we get
\begin{align*}
\Psi_\fq(\Xi)
+\nabla_{\xi'} \Psi_\fq(\Xi)\cdot (\sigma\xi')
&=\Gamma(w_\fm)\cdot (\xi_{\fm}',\xi'')+w_{\fm}\cdot(\sigma\xi')\\
&
=\big(w_{\fm},~\Phi(w_\fm)\big)\cdot (-\mathrm J_\Phi^\intercal(w_\fm)  \xi'', ~\xi'')
+w_{\fm}\cdot(\sigma\xi')\\
&=\big(\sigma w_{\fm},~ -\mathrm J_{\Phi}(w_\fm)w_{\fm}+\Phi(w_\fm)\big)\cdot (\xi',\xi'')\\
&=:\mathbf V_{\fm}\cdot\xi
\end{align*}
where $\Phi=(\phi_1,\phi_2)$.
Hence \eqref{taylor1} can be rewritten as
$\Psi_{\fq}\big(\sigma\xi'+\xi'_\fm,~\xi''\big)
=\mathbf V_{\fm}\cdot\xi+ \sigma^2\overline \Psi(\xi)$
where
\begin{align}\label{taylor2}
\overline \Psi(\xi):=\frac12(\xi')^\intercal \big(\nabla_{\xi'}^2 \Psi_\fq(\xi'_\fm,\xi'')\big) \xi'+\sigma^{-2}\mathcal E(\sigma\xi',\xi'').
\end{align}

Let $\mathcal U:\R^4\rightarrow\R^4$ be the linear map such that  $ \mathcal U(\xi',\xi'')=\big(\sigma\xi'+\xi_\fm',~\xi'')$,
and $\mathcal U^\intercal$ denotes the transpose of $\mathcal U$.
After a change of variables $\xi\rightarrow \mathcal U\xi$ and then $\tau\rightarrow \sigma^2\tau$, we apply \eqref{taylor1} with \eqref{taylor2} so that 
\[
(x,t)\cdot \big(\mathcal U\xi,\sigma^2\tau+\Psi_\fq(\mathcal U\xi)\big)
=\big(\mathcal U^{\intercal}x +t \mathbf V_\fm \big)\cdot \xi+ t\sigma^2\big(\tau+ \overline \Psi(\xi)\big).
\]
If we set $\mathcal L(x,t) = (\mathcal U^\intercal x+t\mathbf V_{\fm},\sigma^2t)$,
we obtain
\begin{align*}
G(x,t) 
&= \iint e^{i(x,t)\cdot (\xi,\, \tau+\Psi_\fq(\xi))} \widehat G(\xi,\, \tau+\Psi_\fq(\xi))
\,d\xi d\tau \\
&= \sigma^4 \iint e^{i \mathcal L (x,t)\cdot (\xi,\, \tau+\overline \Psi(\xi))}
\widehat G\big( \mathcal U \xi,\, \sigma^2(\tau+\overline \Psi(\xi))\big)
\,d\xi d\tau \\
& =  G_{\mathcal U}(\mathcal L(x,t))
\end{align*}
where $\widehat{G_{\mathcal U}}(\xi,\tau) = \sigma^4\widehat G(\mathcal U\xi,\sigma^2\tau)$.

Thus, by the change of variables $(x,t)\to \mathcal L^{-1}(x,t)$,
\eqref{dcp:small scale} is equivalent to 
\begin{align}\label{shifted dcp}
\big\| (G_{\theta(\fm)})_{\mathcal U} \big\|_{L^p(\mathbb R^5)} \le C_\epsilon K^{1-\frac2p+\epsilon} \Big( \sum_{\theta(\fmm)} \| (G_{\theta(\fmm)})_{\mathcal U}\|_{L^p(\mathbb R^5)}^p \Big)^{1/p}.
\end{align}
Here, the Jacobian factor $\det(\mathcal L^{-1})^{1/p}$ is canceled  out. 
Since $\widehat{G_{\theta(\fm)}}$ is supported on a $K^{-1}\sigma^2$-neighborhood of $\mathcal S_{\fq}^{\eta_\circ}(\fm;K^{-1})$ for $\fm \in \mathfrak C(\sigma)$,
the Fourier transform of $(G_{\theta(\fm)})_{\mathcal U }$ is supported in 
\[
\mathcal N_{K^{-1}}( \overline{\mathcal S}_{\fq} ^{\eta_\circ} (\fm;K^{-1}) ) =\Big\{ (\xi',\eta, \overline\Psi (\xi',\eta)) \in \R^5:\,  (\xi',\eta)\in \overline{\mathbf p}^{\eta_\circ}(K^{-1}) \Big\}.
\]
Moreover, the Fourier support of $(G_{\theta(\fmm)})_{\mathcal U}$ are rectangular blocks $\widetilde \theta$ of dimensions  $\sim K^{-\frac12}\times K^{-\frac12}\times K^{-1}\times K^{-1}\times K^{-1}$ which form a partition of $\mathcal N_{K^{-1}}( \overline{\mathcal S}_{\fq} ^{\eta_\circ} (\fm;K^{-1}) )$.

For $\eta$ satisfying $|\eta-\eta_\circ|\le K^{-1}$, we have 
\[
\overline \Psi(\xi',\eta) = \overline \Psi(\xi',\eta_\circ) + O(K^{-1})  \]
by the mean value theorem.\footnote{Note that $ \sigma^{-2}\nabla_{\xi''}\mathcal E(\sigma\xi',\xi'')=O( \sigma|\xi'|^3) $ since $ \mathcal E(\xi)$ is a compactly supported smooth function in $\xi''$.}
   
Thus $\mathcal N_{K^{-1}}( \overline{\mathcal S}_{\fq} ^{\eta_\circ} (\fm;K^{-1}) )$ can be regarded as an $O(K^{-1})$-neighborhood of the cylindrical set   
\[
\Big\{ (\xi',\eta, \overline\Psi (\xi',\eta_\circ)) \in \R^5:\,  (\xi',\eta)\in \overline{\mathbf p}^{\eta_\circ}(K^{-1}) \Big\}  .
\]

Therefore \eqref{shifted dcp} follows by Minkowski's inequality and applying Theorem \ref{dcp-nonv} with $\widehat F$ supported in $\mathcal N_{K^{-1}} (\mathcal S)$ for $\mathcal S = \{ (\xi',\overline\Psi (\xi',\eta_\circ))\}$.    
\end{proof}

\begin{proof}[Proof of Proposition \ref{conic-ext2}]
Let $K=\lambda^{\delta_0}$ for a fixed small constant $0<\delta_0\ll 1$.
By H\"older's inequality, we have
\begin{align}\label{eq001}
\| G\|_p \lesssim K^C \Big(\sum_{\theta} \|G_{\theta}\|_p^p\Big)^{\frac 1p} 
\end{align}
where $\theta=\theta^{\eta_\circ}(\fm;K^{-1})\in \Theta(\sigma_0)$ with
$
\sigma_0=K^{-1},
$
and $\eta_\circ \in K^{-1}\Z^2\cap \{1/2\le |\eta|\le 2\}$.
Note that each $(\xi',\eta)$ contained in $\theta^{\eta_\circ}(\fm;K^{-1})$
satisfies $|\xi'-\xi'_\fm|\le \sigma_0=K^{-1}$ and $|\eta-\eta_\circ|\le K^{-1}$.

Since $\widehat G_\theta$ is supported on a $\lambda^{-1}$-neighborhood, and hence a $K^{-1}\sigma_0^2=\lambda^{-3\delta_0}$-neighborhood, of $\mathcal S_\fq^{\eta_\circ}(\fm;K^{-1}$),
we can apply Lemma \ref{decomp:step} so that 
\begin{align}\label{dcp-intermediate}
\| G_{\theta}\|_p \le C_{\epsilon_1} K^{1-\frac 2{p}+\epsilon_1} \Big( \sum_{\theta_1\subset\theta} \|G_{\theta_1}\|_p^p\Big)^{1/p}
\end{align}
for $\theta_1=\theta^{\eta_\circ}(\fq_1;K^{-1})\in \Theta(\sigma_1)$ with 
\[
\sigma_1=K^{-1/2}\sigma_0=K^{-3/2}.
\]
We observe that $\theta_1$ is contained in a $K^{-1}\sigma_1^2=K^{-4}$-neighborhood of $\mathcal S_\fq^{\eta_\circ}(\fq_1;K^{-1}$).
Repeating the same argument, we can further decompose $\theta_1=\cup \theta_2$ with $\theta_2=\theta^{\eta_\circ}(\fq_2;K^{-1})$ for $\fq_2\in \mathfrak C(\sigma_2)$ with
\[
\sigma_2=K^{-1/2}\sigma_1=K^{-2}.
\]
We then obtain similar bound as in \eqref{dcp-intermediate} for $\theta$, $\theta_1$, and $K$ replaced by $\theta_1$, $\theta_2$, and $K^2$, respectively.

By repeating this process for $N$-steps such that $N <\frac 1{\delta_0}-2$. Then $\sigma_{N+1} < \lambda^{-1/2}\le \sigma_N=K^{-1 -\frac N 2}$ and  
\begin{align}\label{eq002}
\| G_{\theta}\|_p \le C_{\epsilon_1}^N K^{N(1-\frac {2}{p}+\epsilon_1)} \Big( \sum_{\theta_N\subset\theta} \|G_{\theta_N}\|_p^p\Big)^{1/p}
\end{align}
where $\theta_N=\theta^{\eta_\circ}(\fq_N;K^{-1})$ for $\fq_N \in \mathfrak C(\sigma_N)$. 

Finally, we decompose each $G_{\theta_N}$ into functions whose the Fourier transform is supported in $ \theta_\ast :=\theta^{\eta_\circ}(\fq_\ast;K^{-1})$ for $\fq_\ast\in \mathfrak C(\lambda^{-1/2})$.
By H\"older's inequality, we have 
\begin{equation}\label{eq003}
\| G_{\theta_N}\|_p \lesssim \lambda^{\frac{ 2\delta_\circ} {p'}} (\sum_{\theta_\ast} \| G_{\theta_\ast} \|_{L^p}^p)^{1/p} 
\end{equation}
since the number of $\theta_\ast\subset \theta_N$ is at most $ (\sigma_N\lambda^{1/2})^2 = \lambda^{-2(1+\frac N 2)\delta_0 + 1} < \lambda^{ 2\delta_0} $.  
Hence, by combining \eqref{eq001}, \eqref{eq002}, and \eqref{eq003}, we obtain 
\begin{align*} 
\big\|G \big\|_{L^p(\mathbb R^5)} \le  C_{\epsilon_1}^N  K^{ C + N (1-\frac2p +\epsilon_1)}  \lambda^{\frac{ 2\delta_0} {p'}} \Big( \sum_{\theta(\fm)} \| G_{\theta(\fm)}\|_{L^p(\mathbb R^5)}^p \Big)^{1/p},
\end{align*}
which implies \eqref{prop_cone_result} by taking small $\delta_0, \epsilon_1 \le \epsilon/100$.
\end{proof}

Now, we prove Corollary \ref{cor:dec4} by applying \eqref{extension decoupling norm} which is equivalent to  \eqref{prop_cone_result}.  

\begin{proof}[Proof of Corollary \ref{cor:dec4}]
By Lemma \ref{lem:away_crit}, Lemma \ref{tri-L2} and Lemma \ref{average-sum}, it is enough to show the following holds: suppose that $\supp \widehat f \subset \mathbb A_\lambda$, then
\begin{equation}\label{Tpp}
\big\|T_\fq [\psi] \widehat {f_{\mathbf P_\lambda(\fq)}} \big\|_{L^p(B^5(0,10))}
\lesssim \lambda^{1-\frac 3p+\epsilon}   \|f \|_p .
\end{equation}

By change of variables $\xi \to \lambda\xi$ and $x\rightarrow \lambda^{-1}x$, we get
\begin{equation}\label{5p}
\big\|T_\fq [\psi] \widehat {f_{\mathbf P_\lambda(\fq)}} \big\|_{L^p(B^5(0,10))} = \lambda^{-\frac 5 p }  \big\|T_\fq [\widetilde\psi] \widehat { f^\lambda_{\mathbf P_1(\fq)}} \big\|_{L^p(B^5(0,10\lambda))}
\end{equation}
where $\widehat{f^\lambda_{\mathbf P_1(\fq)}}=\lambda^4 \widehat{f_{\mathbf P_\lambda(\fq)}}(\lambda\cdot)$, and $\widetilde\psi$ is supported in $\mathbb A_1$. 

 For $\mathbf p \in \mathcal P(\lambda^{-\frac12})$ defined in \eqref{set of plank}, we set $g_{\mathbf p} = g \chi_{\mathbf p}$ for a smooth function $\chi_{\mathbf p}$ such that $\chi_{\mathbf p}^2 =1$ on $\mathbf p$ and supported on $2\mathbf p$.
By Theorem \ref{extension decoupling}, we obtain
\[
\big\|T_\fq [\widehat\psi]\widehat { f^\lambda_{\mathbf P_1(\fq)}} \big\|_{L^p(B_\lambda )} \lesssim \lambda^{1-\frac2p+\epsilon} \big(\sum_{\mathbf p \in \mathcal P(\lambda^{-\frac12})} \big\|  T_{\fq}[\widetilde\psi] (\widehat { f^\lambda_{\mathbf P_1(\fq)}}\chi_{\mathbf p}) \big\|_{L^p(W_{B_\lambda})}^p \big)^{\frac1p}.
\]
The kernel of $T_{\fq}[\widetilde\psi]$ is given by
\[
K_{\mathbf p} (x,t) = \int e^{ix\cdot \xi-it\Psi_{\fq}(\xi)} \widetilde\psi(\xi)\chi_{\mathbf p}(\xi)\,d\xi.
\]
Let $b_1,b_2$ be the unit vectors in $\R^4$ aligned along the short axes of $\mathbf p$. 
Let $b_3,b_4$ be unit vectors perpendicular to $b_1, b_2$ so that $\{b_1,\dots,b_4\}$ spans $\R^4$.
By repeated integration by parts, we see
\[
|K_{\mathbf p} (x,t)| \lesssim \lambda^{3}\big(1+\sum_{j=1}^2 \lambda^{\frac12}| b_j \cdot (x+t\nabla\Psi_{\fq}(\Xi)(w_{\fm})|
+\sum_{j=3}^4 \lambda| b_j \cdot (x+t\nabla \Psi_{\fq}(\Xi) ) | \big)^{-N}
\]
for sufficiently large $N\ge1$ where $\Psi_{\fq}(\Xi)$ is the the center in $\mathbf p =\mathbf p(\fm) \in \mathcal P(\lambda^{-1/2})$.
Since $\|K_{\mathbf p}\|_1\lesssim 1$, we reverse the change of variable $\xi \to \lambda^{-1}\xi$ to obtain 
\begin{align*}
\big\|T_\fq [\widehat\psi]\widehat { f^\lambda_{\mathbf P_1(\fq)}} \big\|_{L^p(B_\lambda )}
&  \lesssim \lambda^{1-\frac2p+\epsilon} \big(\sum_{\mathbf p \in \mathcal P(\lambda^{-\frac12})} \big\|   f^\lambda_{\mathbf P_1(\fq)}\ast \chi^\vee_{\mathbf p}    \big\|_{p}^p \big)^{\frac1p}   \\
& \lesssim \lambda^{1-\frac 2p+\frac 4 p+\epsilon} \big(\sum_{\mathbf \fm} \|f_{\mathbf p(\fm)}\|_p^p\big)^{\frac1p}.
\end{align*}
By \eqref{5p} and \eqref{little sum}, we obtain \eqref{Tpp}.
\end{proof}

\section{Proof of Theorem \ref{LS}}\label{sec:completion}

Theorem \ref{LS} is a special case of Theorem \ref{LS-com} which is a consequence of Lemma \ref{lem:fast_dec} and Theorem \ref{thm-LS}. 
In this section, we first give a proof of Proposition \ref{multi-interpol}, and then complete the proof of Theorem \ref{thm-LS}.

\begin{proof}[Proof of Proposition \ref{multi-interpol}]
Proposition \ref{tri-L2} and Corollary \ref{cor:dec4} give a local version of \eqref{tri-result}.
We extend it to a global one by using the fact that the kernel of $\mathcal A_\lambda[\varphi_\fq] f$ rapidly decays away from $B^5(0,10)$. 
This can be done by the typical localization argument (e.g., see \cite{HKL2, Lee1}).

Let us set 
\[
\mathcal A_\lambda  [\varphi_\fq] f(x,t) = K_\lambda (\cdot, t) \ast f (x),
\]
where 
\[
K_\lambda(x,t)= \chi_I(t)\iint e^{i (x - t\Gamma(u,v))\cdot \xi}  \varphi_\fq  ( u,v) du dv \beta_\lambda( \xi ) d\xi .
\]
By the Fourier inversion formula followed by a change of variables $\xi \rightarrow \lambda \xi$, we have
\begin{align}\label{kkk}
K_\lambda(x,t)=\chi_I(t)\lambda^4\int \widehat \beta_1(\lambda(x+t\Gamma(u,v)))\varphi_\fq(u,v)\,dudv.
\end{align}
It is obvious that $|K_\lambda|\lesssim \lambda^2$, which implies 
\begin{equation}\label{01}
\|  \mathcal A_\lambda [\varphi_{\fq }] f  \|_{L^\infty(\mathbb R^5)}
\lesssim \lambda^2  \|f \|_1 .
\end{equation}
By H\"older's inequality and interpolation between \eqref{01}, \eqref{dec4}, \eqref{centered}, and the trivial $L^\infty$ bound, we obtain 
\[
\Big\|\prod_{j=1}^3\big|\mathcal A_\lambda [\varphi_{\fq_j}]f_j \big|^{\frac13} \Big\|_{L_{x,t}^q (B^5(0,10))}
\le C  \lambda^{\gamma}\prod_{j=1}^3\|f_j \|_p^{\frac13} 
\]
for $\gamma > \frac 2 p -\frac 5 q$ and  for $\frac 1 q \le \min\{ \frac 1p,\, \frac 1{3p}+\frac16,\, \frac23(1-\frac1p)\}$.

To extend this to a global one,   we observe that
\begin{align}\label{ker2}
|K_\lambda(x,t)|
\lesssim
\lambda^4(1+\lambda|x+t\Gamma(u,v)|)^{-N}
\lesssim
\lambda^{4-N}(1+|x|)^{-N}
\end{align}
provided that $|x|\ge c^* =4 \sup_{(u,v)\in\fq}|\Gamma(u,v)|$.

Considering lattice unit cubes $B$ in $\R^4$, we denote by $\widetilde B$ the square with the side length $c^*$ with the same center as $B$. 
So, by the kernel estimates \eqref{ker2} and Young's inequality, we have $\|\mathcal A_\lambda [\varphi_{\fq}](f\chi_{ {\widetilde B}^{ c}}) \|_q \lesssim \lambda^{-N}\|f\|_p$ for $1\le p \le q \le \infty$ and for any $N\ge1$.
Hence, we obtain 
\begin{align*}
\int_{\R^4\times I}\prod_{j=1}^3 |\mathcal A_\lambda [\varphi_{\fq_j}] f_j|^{\frac q3}
& =\sum_B
\int_{B\times I}\prod_{j=1}^3 |\mathcal A_\lambda [\varphi_{\fq_j}] (f_j\chi_{\widetilde B})|^{\frac q3}
+O(\lambda^{-N})\prod_{j=1}^3\|f_j\|_p^{\frac q3} \\ 
& \lesssim \lambda^{\gamma}
\sum_B \prod_{j=1}^3 \|f_j \chi_{\widetilde B}\|_p^{\frac q3} +\lambda^{-N} \prod_{j=1}^3
\|f_j\|_p^{\frac q3}.
\end{align*}
By taking $N$ sufficiently large , we apply H\"older's inequality and use the embedding $\ell^{p/3} \subset \ell^{q/3}$ 
to derive the desired estimate \eqref{tri-result}.
\end{proof}

Now we are ready to prove Theorem \ref{thm-LS}.

\begin{proof}[Proof of Theorem \ref{thm-LS}]
By the multilinear decomposition in Lemma \ref{lem-multi-decomp}, we have
\begin{equation*}
\big\| \mathcal A_\lambda [\varphi]f \big\|_{q} 
\lesssim 
\big\|\max_{\fq}|\mathcal A_\lambda [\varphi_{\fq}]f|\big\|_{q} + C_K
\max_{\substack{(\fq_1,\fq_2,\fq_3) \\ \in \mathfrak D_{K^{-1}}}}\Big\| \prod_{j=1}^3 |\mathcal A_\lambda [\varphi_{\mathfrak q_j}]f|^{\frac13} \Big\|_{q} .
\end{equation*}
Combining \eqref{linear 1} and \eqref{tri-result}, we obtain
\begin{align*}
\big\| \mathcal A_\lambda [\varphi]f \big\|_{L^q(\R^4\times I)} 
\lesssim ( K^{-2(1+\frac3q-\frac3p)}
\mathscr Q (K^{-2} \lambda)+C_N\lambda^{-N}+ C \lambda^{ \gamma } )\| f\|_p,
\end{align*}
where $\gamma> \frac2 p-\frac 5 q$ and $\frac 1 q \le \min\{ \frac 1p,\, \frac 1{3p}+\frac16,\, \frac23(1-\frac1p) \}$.  
Taking supremum over $f$ with $\|f\|_{L^p}\le 1$ and $N\ge-\gamma$, it follows that 
\begin{equation}\label{QQ}
\mathscr Q(\lambda) \lesssim K^{-2(1+\frac3q-\frac3p)}
\mathscr Q (K^{-2} \lambda)+ C \lambda^{ \gamma }. 
\end{equation}

Let us define
\[
\mathfrak Q(\lambda):=\sup_{1\le R\le \lambda}R^{-\gamma }\mathscr Q(R).
\]
Then 
\[(K^{-2}\lambda)^{-\gamma} \mathscr Q (K^{-2} \lambda)  \le \mathfrak Q (\lambda ')\] for $\lambda' \ge K^{-2} \lambda \ge 1$. 
Otherwise, if $\lambda < K^2$, then we have $ \mathscr Q (K^{-2} \lambda) \lesssim K^C$ for some $C>0$. 
So,  \[(K^{-2}\lambda)^{-\gamma} \mathscr Q (K^{-2} \lambda)  \le \mathfrak Q (\lambda ') + K^C.\] 
Multiplying $\lambda^{-\gamma}$ to both sides of \eqref{QQ}, we get
\[
\lambda^{-\gamma}\mathscr Q(\lambda) \lesssim K^{-2(1+\frac3q-\frac3p +\gamma)}  \mathfrak Q (\lambda') + C   K^C
\]

It follows by taking supremum over $\lambda \le \lambda'$ that 
\[
\mathfrak Q(\lambda') \lesssim K^{-2(1+\frac3q-\frac3p +\gamma)}  \mathfrak Q (\lambda' ) + C K^C. 
\]
Since $\gamma >\frac 2 p -\frac 5 q$, we have  $ 1+\frac3q-\frac3p +\gamma > 1-\frac 2 q - \frac 1 p\ge 0$.
We can choose an appropriate constant $K\ge 1$ so that $K^{-2(1+\frac3q-\frac3p +\gamma)} < 1/2$, which implies   $ \mathfrak Q(\lambda') \lesssim C K^C$ for $\frac 1 q \le \min\{ \frac 1p,\, \frac 1{3p}+\frac16,\, \frac23(1-\frac1p) \}$ and $ \frac 2 q \le 1-\frac 1p  $. 
This completes the proof.  
\end{proof}

\section{Proof of Theorem  \ref{n-dim-thm}}\label{sec:proof of high}

We follow the classical argument in \cite{Lee} proving the spherical maximal function, while 
we apply Greenleaf's $L^2$ Fourier restriction estimate in place of the Strichartz estimate.
For the restricted weak type estimates, we use the Bourgain's interpolation lemma (see e.g., \cite{B3,CSW,Lee}), as follows. 
\begin{lem}\label{Bourgain trick}
Let $1\le p_1,p_2,q_1,q_2 \le \infty$ and $\epsilon_1,\epsilon_2>0$. Suppose that $T_k$, $k\in \Z$ are sublinear operators maps $L^{p_j}$ to $L^{q_j}$, $j=1,2$ such that
\[
\big\| T_k  \big\|_{L^{p_1}\to L^{q_1}} \le M_1 2^{   \epsilon_1 k }~ \text{ and }~  \big\| T_k  \big\|_{L^{p_2}\to L^{q_2}} \le M_2 2^{ -  \epsilon_2 k} 
\]
for some $\epsilon_1, \epsilon_2 >0$. 
Then 
\[\|\sum_k T_k g \|_{q,\infty}\le M_1^{\theta} M_2^{1-\theta} \|f\|_{p,1}\] for  $\theta = \frac{\epsilon_2}{\epsilon_1+\epsilon_2}$, $\frac 1p = \frac{\theta}{p_1}+\frac{1-\theta}{p_2}$,  $\frac 1q=\frac{\theta}{q_1}+\frac{1-\theta}{q_2}$.
\end{lem}

\begin{proof}[Proof of Theorem \ref{n-dim-thm}]
We show the restricted weak type estimate  
\begin{align}\label{rwt} \big\| \mathcal M_\Gamma f \big\|_{L^{q,\infty}(\mathbb R^{2n})} \le  C\|f\|_{L^{p,1}(\mathbb R^{2n})}  \end{align}
holds at 
$(1/p,1/q) = \mathrm P_1, \mathrm P_2, \mathrm P_3$ where
\[
\mathrm P_1 =\Big(\frac{2n}{3n+2},\,\frac n{3n+2}\Big), \quad
\mathrm P_2 = \Big(\frac{2n-1}{3n-1},\,\frac{n}{3n-1}\Big),\quad
\mathrm P_3 = \Big(\frac n{n+1},\, \frac {n}{n+1}\Big).
\]
Then, \eqref{supA} for $(\frac1p,\frac1q)\in \mathcal W_n$ follows by applying weak-type Marcinkiewicz interpolation between those estimates at $\mathrm P_1, \mathrm P_2, \mathrm P_3$ and the trivial bound at $\mathrm O=(0,0)$.

As in the preceding discussion in Section \ref{sec:pre}, we consider the contribution from the operator $\sup_{t\in [1,2]}|\mathbf A_\lambda [\varphi]   f(x,t)|$ for $\lambda \ge 1$, where $\mathbf A_\lambda = \mathbf A P_\lambda$ with the standard Littlewood-Paley projection operator denoted by $\widehat{P_\lambda  f} =\beta_\lambda \widehat f$. 
We claim that 
\begin{align}
\big\| \sup_{t\in I} |\mathbf A_\lambda[\varphi] f (\cdot ,t) | \big\|_2 &\lesssim \lambda^{-\frac {n-1}2}\|f\|_2, \label{M22}\\
\big\| \sup_{t\in I} |\mathbf A_\lambda[\varphi] f (\cdot ,t)| \big\|_{\infty} &\lesssim \lambda^n\|f\|_1, \label{Minf}\\
\big\| \sup_{t\in I} |\mathbf A_\lambda[\varphi] f (\cdot ,t)|\big\|_1 &\lesssim \lambda \|f\|_1, \label{M11}\\
\big\| \sup_{t\in I} |\mathbf A_\lambda[\varphi] f (\cdot ,t)| \big\|_{\frac{2(n+2)}{n}} &\lesssim 
\lambda^{-\frac{n(n-2)}{2(n+2)}}\|f\|_2. \label{Mstri}
\end{align}

Then the estimate \eqref{rwt} can be derived by Lemma \ref{Bourgain trick}. In fact, \eqref{rwt} at $\mathrm P_2$ (or $\mathrm P_3$) follows by interpolating between \eqref{M22} and \eqref{Minf} (or \eqref{M11}), respectively.

For $\mathrm P_1$, we consider the cases $n\ge 4$ and $n=2$ separately. 
When $n\ge 4$, interpolation between \eqref{Mstri} and \eqref{Minf} gives  \eqref{rwt} at $\mathrm P_1$.
Since there is no decaying property in \eqref{Mstri} for $n=2$, we will discuss it at the end of the proof.

It remains to prove \eqref{M22}, \eqref{Minf}, \eqref{M11} and \eqref{Mstri}.
The kernel $\mathbf K_\lambda$ of $\mathbf A_\lambda[\varphi] $ is given by
\begin{align*}
\mathbf K_\lambda(x,t) & =\chi_I(t) \int e^{i  x\cdot \xi} \int e^{ -i  t \Gamma(\mathbf u) \cdot \xi}\varphi(\mathbf u)d\mathbf u \, \beta_\lambda(\xi)\,d\xi \\ 
& = \lambda^{2n}\chi_I(t)\int \widehat{\beta_1} (\lambda(x - t \Gamma(\mathbf u))) \varphi(\mathbf u) d\mathbf u.
\end{align*} 
Note that $|\widehat{\beta_1}(\lambda(x - t\Gamma(\mathbf u )))|\lesssim (1+\lambda|x - t\Gamma(\mathbf u )|)^{-N}$ for any $N\ge 1$.  
For each fixed $t$, the surface $t\Gamma(\mathbf u)$ can be covered by $O(\lambda^{n/2})$ rectangles  $\theta$
of dimensions $(\lambda^{-1/2})^n\times (\lambda^{-1})^n$. Therefore, the integral  $|\int \widehat{\beta} (\lambda(x - t \Gamma(\mathbf u))) \varphi(\mathbf u) d\mathbf u|$ is bounded by $\lambda^{-n}$, which implies
\[\| \mathbf K_\lambda\|_{L^\infty_{x,t}} \lesssim \lambda^n.\] 
Moreover, the support of $\sup_{t\in I} \mathbf K_\lambda(x,t)$ in $x$-variable is essentially contained in a $\lambda^{-1}$-neighborhood of $\mathop{\cup}\limits_{t} t\Gamma$. This yields
\begin{align*}
\| \mathbf K_\lambda \|_{L^1_xL^\infty_t }
& = \lambda^{2n} \,\Big\| \int \widehat{\beta} (\lambda(x - t \Gamma(\mathbf u))) \varphi(\mathbf u) d\mathbf u \Big\|_{L^1_xL^\infty_t } \\
&\lesssim \lambda^{2n} \times |\theta| \times~ \text{number of}~\theta \times ~|\text{support of}~  x |~ \\
&= \lambda^{2n} \times  \lambda^{-n/2 - n}  \times \lambda^{n/2}\times \lambda^{-(n-1)} = \lambda. 
\end{align*}  
Applying Young's convolution inequality, we get \eqref{Minf} and \eqref{M11}.

Next, we turn to prove \eqref{M22}.
By H\"older's inequality followed by the fundamental theorem of calculus, it is well-known  that for $F \in C^1(\mathbb R)$ and $q>1$, 
\begin{equation}\label{emb}
  \sup_{t\in I } |F(x,t)|   \lesssim  \ell^{-1/q} \big(\int_I |F(x,t)|^q dt \big)^{1/q}   + \ell^{1/q'} \big(\int_I | \partial_t F(x,t) |^q dt \big)^{1/q} 
\end{equation}
 for each $\ell\le |I|$.  (For example, see \cite[Lemma 1(p.499)]{Stein}.)
Taking $\ell=\lambda^{-1}$, we obtain
\begin{align}\label{emb2}
\| \sup_{t\in I} |\mathbf A_\lambda[\varphi]f(\cdot,t)| \|_{L^q_x}    \lesssim \lambda^{\frac1q}\|\mathbf A_\lambda[\varphi]f \|_{L^q_{x,t}}   + 
\lambda^{-\frac1{q'}}\|\partial_t \mathbf A_\lambda[\varphi]f \|_{L^q_{x,t}}.
\end{align}

We write $\mathbf A_\lambda[\varphi]f (x,t) = (\mathrm m(t\xi) \beta_\lambda(\xi)\widehat f(\xi))^\vee$, where  
\[
  \mathrm m ( t\xi)   :=  \int e^{i t \Gamma(\mathbf u)\cdot \xi} \varphi(\mathbf u)\,d\mathbf u .
\]
Note that the multiplier of $\partial_t \mathbf A_\lambda[\varphi]$ has a similar form of $\mathrm m$ but $\varphi$ is replaced by $\varphi(\mathbf u)\big(\Gamma(\mathbf u)\cdot \xi\big)$. By the Mikhlin multiplier theorem, it suffices to estimate $L^q$-norm of $\mathbf A_\lambda[\varphi]f$. 

By the assumption \eqref{curv-high}, we have  the decay estimate $|\partial^\alpha \mathrm m (t\xi)| \lesssim |\xi|^{-n/2-|\alpha|}$ for all $t\in I$. 
Applying Plancherel's theorem, together with \eqref{emb2} for $q=2$, we obtain
\[
\| \sup_{t\in I} |\mathbf A_\lambda[\varphi]f(\cdot,t)| \|_{L^2(\mathbb R^{2n})}\lesssim \lambda^{-(n-1)/2} \| f\|_{L^2},
\]
which gives \eqref{M22}.

To prove \eqref{Mstri}, it is enough to show
\begin{equation}\label{Stri}
\|   \mathbf A_\lambda[\varphi]f\|_{L^{\frac{2(n+2)}{n}}} \lesssim \lambda^{-\frac{n(n-1)}{2(n+2)}} \|f\|_{L^2} . 
\end{equation}
Then \eqref{Mstri} follows by applying \eqref{emb2} with $q = \frac{2(n+2)}{n}$  and  $\|  \partial_t \mathbf A_\lambda[\varphi]f\|_{\frac{2(n+2)}{n}} \lesssim \lambda^{-\frac{n(n-1)}{2(n+2)} + 1} \|f\|_{2}$.

Since Lemma \ref{lem:fast_dec} remains to be true for any $n\ge 2$ by taking $c_\ast = \| \frac{\partial(\phi_1,\phi_2,\dots,\phi_n)}{\partial(u_1,u_2,\dots,u_n)}\|$, we may assume that $\widehat{f }$ is supported in 
\[
\mathbb A_\lambda^n = \{ \xi=(\xi',\xi'') \in \mathbb R^{2n}:2^{-1}\le |\xi''| <2\lambda, ~ |\xi'|\le 4c_\ast|\xi''| \},~\quad \lambda \ge 1 ,
\]
where $\xi'=(\xi_1,\dots,\xi_n)$ and $\xi''=(\xi_{n+1},\dots,\xi_{2n})$. 
Decomposing the domain of $\Gamma$ into small cubes, say $\fq$ as in \eqref{decomp1}, we see Lemma \ref{lem:away_crit} is valid in higher dimensions. 
As in the proof of Lemma \ref{average-sum}, there exists $\mathbf z(\xi) = (u_1(\xi), \dots, u_d(\xi))$ such that 
\begin{align*}
\partial_{u_j}\big( \Gamma( \mathbf z(\xi) ) \cdot \xi   \big) = 0, \quad \forall j=1,\dots,n ,
\end{align*}
by the assumption \eqref{curv-high} (cf. \eqref{nablazero}).

For $ \Psi_\fq^n (\xi) : = \Gamma(\mathbf z(\xi))\cdot \xi $ (cf. \eqref{phase-multiplier}), we define the extension operator
\[
\mathbf T_{\fq}^n [\psi] g(x,t) = \int e^{i  (x\cdot \xi - t \Psi_\fq^n (\xi)) } \psi(\xi) g(\xi) d\xi ,
\]
where $\psi$ is a compactly supported smooth function. 
Note that $\Psi_\fq^n$ is homogeneous of degree $1$ and $\rank \nabla_\xi^2 \Psi_\fq^n \ge n$ by the same argument as in the proof of Lemma \ref{rank-hess-2}. 
In other words, at each point of the domain of $ \Psi_\fq^n $, at least $n$ principal curvatures are nonzero. 
Thus we have
\begin{equation}\label{greenleaf}
\| \mathbf T_{\fq}^n [\psi] g\|_{L^{\frac{2(n+2)}{n}}(\mathbb R^{2n+1})} \le C \| g\|_{L^2(\mathbb R^{2n})} ,
\end{equation}
by the following $L^2$ Fourier restriction estimate due to Greenleaf \cite{Gr81}:
\begin{thm}[{\cite[Corollary 1]{Gr81}}]
Let $S\subset \R^{d+1}$ be a smooth hypersurface with a smooth measure $d\mu$ supported away from the boundary. Suppose that at each point of $\supp (d\mu)$, at least $n$ principal curvatures are nonzero. Then,  
\[
\big\| \widehat f \big|_S \big\|_{L^2(d\mu)} \lesssim \|f\|_{\frac{2(n+2)}{n+4}} .
\]
\end{thm}

Modify the proof of Lemma \ref{average-sum}, for a sufficiently large integer $N$, 
\begin{align*} 
\mathbf A_\lambda[\varphi_{\fq}] f_{\mathbf P(\fq)}(x,t)   = \chi_I(t) \sum_{l=0}^{N+5} t^{-1-l}   \mathbf T_\fq^n [\mathbf b_{\fq,l}|\xi|^{-\frac n2-l}] \widehat {f_{\mathbf P(\fq)}}   (x,t)+  \mathcal E^n  f   (x,t),
\end{align*}
 where  $\mathbf b_{\fq,l} (\xi)    $ is a smooth function such that $|\partial_\xi^\alpha \mathbf b_{\fq,l} | \lesssim |\xi|^{-|\alpha|}$, and $\|  \mathcal E^n   f  \|_q \le \lambda^{-N}\|g\|_p$ for any $1 \le p \le q \le \infty$.  
Hence the desired estimate \eqref{Stri} follows by
applying \eqref{greenleaf} with a change of variables $\xi\to\lambda\xi$ to get 
\begin{align*}
\| \mathbf T_\fq^n [\mathbf b_{\fq,l}|\xi|^{-\frac n2-l}] \widehat {f_{\mathbf P(\fq)}}  \|_{L^q} 
&\le \lambda^{-\frac n2 - l} \lambda^{-\frac{n(2n+1)}{2(n+2)}} \| \lambda^{2n} \widehat {f_{\mathbf P(\fq)}}(\lambda\cdot)\|_{L^2(\mathbb R^{2n})} \\
& = \lambda^{-l} \lambda^{ -\frac{n( n-1)}{2(n+2)}} \| f\|_{L^2(\mathbb R^{2n})} .  
\end{align*}

\noindent(The case of $n=2$.)
Note that \eqref{M22}--\eqref{Mstri} is valid for $n=2$.
Due to the decaying property in \eqref{M22}, we apply Lemma \ref{Bourgain trick} to \eqref{M22}, \eqref{Minf}, and \eqref{M11} in order to obtain the restricted weak type estimate \eqref{rwt} at $P_2=(\frac35,\frac25)$, $P_3=(\frac23,\frac23)$. 
Then the strong type bounds of $\sup_{t\in I}|\mathbf A_\lambda[\varphi]f|$ on the open segment $(P_2, P_3)$ follows by interpolation.

On the other hand, since there is no decaying property in \eqref{Mstri}, Lemma \ref{Bourgain trick} cannot be applied. 
Instead, we interpolate $L^2$ and $L^4$ estimates for the averaging operator $\mathbf A_\lambda[\varphi]f$. 
By Plancherel's theorem and \eqref{Stri} with $n=2$, we have $ 
\| \mathbf A_\lambda[\varphi]  \|_{L^2\to L^2} \lesssim \lambda^{-1} $ and $
\| \mathbf A_\lambda[\varphi] \|_{L^2 \to L^4} \lesssim \lambda^{-\frac 14} $.
Then  $\| \mathbf A_\lambda[\varphi] \|_{L^2 \to L^q} \lesssim \lambda^{\frac12-\frac 3q} $ for $2 \le q \le 4$ by interpolation.
By \eqref{emb2}, it follows that 
\[
\| \sup_{t\in I} |\mathbf A_\lambda[\varphi] f|  \|_{L^q } \lesssim \lambda^{\frac12-\frac 2q} \|f\|_{L^2}.  
\]
When $q < 4$ i.e. $\frac 12 -\frac 2q<0$, the strong type estimate \eqref{supA}  holds for $q <4$ and $p=2$.
Therefore, by interpolation with the trivial $L^\infty$ bound and the strong type bound on $(P_2, P_3)$ previously obtained, we see that \eqref{supA}  holds in  $\mathcal W_2 \setminus  \{ (\mathrm O, \mathrm P_1], (\mathrm P_1, \mathrm P_2] \}$.  
 \end{proof}

\begin{rmk} 
To obtain the restricted weak type bound at $(\frac 12,\frac14)$, it seems to be necessary to rely on local smoothing estimates without $\epsilon$-loss on regularity. 
Following the argument in \cite{Lee}, one may need to establish sharp bilinear restriction estimates associated with $\Psi_q$. We do not pursue this here. 
\end{rmk}

\section{Necessary conditions}\label{sec:sharpness}

In this section, we prove the ranges of $(1/p,1/q)$ in Theorem \ref{n-dim-thm} and \ref{thm:rank1} are optimal, and discuss the sharpness of the regularity in Theorem \ref{LS}.

\begin{prop}\label{prop:nece-high}
Let $\mathbf u = (u_1,\dots,u_n) \in \mathbb R^n$, $n\ge 2$. Suppose that $\Gamma(\mathbf u)=(\mathbf u,\Phi(\mathbf u))$ is a nondegenerate submanifold that satisfies \eqref{curv-high}.  
Then \eqref{supA} holds only if  
\[
\frac 1p\le \frac {2}{q}, \quad \frac{3n}{2p} \le  \frac n2+\frac{3n-2}{2q}, \quad \frac{2n}p \le n+ \frac{n-1}q.
\]
\end{prop}

\begin{proof}
(a) $\frac 1p \le \frac{2}q$:
Let $f=\chi_{\mathcal N_\delta(\Gamma)}$ where $\mathcal N_\delta(\Gamma)$ is a $\delta$-neighborhood of the surfaces $\Gamma$.
Then $\|f\|_p \lesssim \delta^{\frac np}$. Moreover $|\mathbf A[\varphi]f(x,t)|\gtrsim 1$ for $(x,t)\in B^n(0,\delta)$. Then $\delta^{\frac {2n}q} \lesssim \delta^{\frac np}$. Taking a arbitrary small $\delta$, we have $\frac1p \le \frac 2q$.

\smallskip

\noindent(b) $\frac{3n}{2p} \le \frac{3n-2}{2q}+\frac n2$: 
By the Taylor expansion, we have
\[
\Gamma(\mathbf u) - \Gamma(\mathbf u_\circ) = \big(\mathbf u-\mathbf u_\circ,\Phi(\mathbf u)-\Phi(\mathbf u_\circ)\big)
=\sum_{i=1}^n (\mathbf e_i,\partial_{u_i}\Phi(\mathbf u_\circ))(u_i-1)+O(|\mathbf u-\mathbf u_\circ|^2) ,
\]
where $\mathbf e_i's$ are the standard unit vectors in $\R^n$ and $\mathbf u_\circ=(\frac 12,\dots,\frac12) \in \R^n$.
Let $f=\chi_{U_n}$ where $U_n$ is the parallelepiped in $\R^{2n}$ centered at $0$ of dimensions $(\delta^{\frac12})^n\times(\delta)^n$ such that the long axes are in the direction of $(\mathbf e_{i}, \partial_{u_i}\Phi(\mathbf u_\circ))$, $i=1,\dots,n$.  
Then  $\sup_t|\mathbf A[\varphi] f(x,t)|\gtrsim \delta^{\frac n2}$ for $x\in V_n$
where \[
V_n=\bigcup_t V_{n}(t) := \bigcup\big\{t(\mathbf u,\Phi(\mathbf u)):|\mathbf u-\mathbf u_\circ|\le \delta^{\frac12}, 1\le t\le2\big\}.
\]

Since  $|V_n|=\delta^{\frac n2}\times \delta^{n-1}$,
\[
\delta^{\frac n2+\frac{3n-2}{2q}}\lesssim \|\sup_t| \mathbf A[\varphi] f(\cdot,t)| \|_q \lesssim \|f\|_p \lesssim \delta^{\frac {3n}{2p}}.
\]
Letting $\delta\rightarrow 0$ gives $\frac{3n}{2p} \le \frac{3n-2}{2q}+\frac n2$.

\smallskip

\noindent (c) $\frac{2n}p \le n+\frac{n-1}q$: We choose $f=\chi_{B(0,\delta)}$. 
For each fixed $t$, $|\mathbf A[\varphi] f(x,t)|\gtrsim \delta^{n}$ if $x \in N_\delta(E_n(t)) $  for
\[
E_n(t)=\{t(\mathbf u,\Phi(\mathbf u)): |\mathbf u|\le \epsilon\} \] 
with a small $\epsilon>0$.
Then $\sup_{t \in I}|\mathbf A[\varphi] f(x,t)|\gtrsim \delta^n$ for $x\in E_n:=\bigcup_{t \in I} E_n(t)$. 
Since $|E_n|\sim(1)^n\times 1\times (\delta)^{n-1}$, we have
 $\delta^{n}\delta^{\frac{n-1}q} \lesssim \|\sup_{t \in I}|\mathbf A[\varphi] f|\|_q \lesssim \|f\|_p \lesssim \delta^{\frac{2n}p}$.
 Taking $\delta \rightarrow0$, we obtain $\frac{2n}p \le n+\frac{n-1}q$.
\end{proof}

We discuss the sharpness of Theorem \ref{thm:rank1}.

\begin{prop}\label{nece:rank1}
Let $\Gamma=(u,v,\phi_1(u,v),\phi_2(v))$ such that $\partial_u^2\phi_1 \neq0$ and $\partial_v^2\phi_2\neq0$. Then \eqref{supA} holds only if
\begin{align}\label{pqrange1}
\tfrac 1p \le \tfrac 2q, \quad \tfrac 3p \le 1+\tfrac 1q, \quad \tfrac 1p \le \tfrac12.
\end{align}
\end{prop}

\begin{proof}
(a) $\frac 1p\le \frac 2q$. The proof is identical to Proposition \ref{prop:nece-high}. We omit the detail.

\smallskip

\noindent (b) $\frac3p\le 1+\frac1q$. 
Without loss of generality, we write 
\[
(\phi_1(u,v),\phi_2(v))=\big(u^2+\mathcal E_1(u,v),~v^2+\mathcal E_2(v)\big)
\]
where $\mathcal E_1 = O(|(u,v)|^3)$ and $\mathcal E_2  = O(|v|^3)$.  

Let us choose $f=\chi_{U'}$ where 
\begin{align*}
U'=\{x_1\mathbf e_1+ x_2\mathbf u_2&+x_3\mathbf e_3+x_4\mathbf u_4 \in \R^4 \,:\,  \\
 &|x_1|\lesssim 1,\,|x_2|\le \delta^{\frac12},\,|x_3| \lesssim 1,\,||x_4|-1|\le \delta\}.
\end{align*}
Here, $\mathbf u_2=(0,1,0,\partial_v\phi_2(0))$ and $\mathbf u_4=(0,-\partial_v\phi_2(0),0,1)$, and $\mathbf e_1, \mathbf e_3$ are the standard unit vectors. 
It follows that $\sup_{t\in I}|\mathcal Af(x,t)|\gtrsim \delta^{\frac 12}$ holds for $x\in V'$, where 
\begin{align*}
V'=\{x_1\mathbf e_1+x_2\mathbf u_2&+x_3\mathbf e_3+x_4\mathbf u_4\in \R^4: \\
&|x_1|\le \epsilon,\,|x_2|\le \delta^{\frac12},\, |x_3|\le \epsilon,\,1-\epsilon \le x_4 \le 1+\epsilon\}
\end{align*}
for a small $\epsilon>0$. Thus $|V'|\sim \delta^{\frac12}$ and therefore, $\delta^{\frac 12+\frac 1{2q}} \lesssim \| \sup_t|\mathcal Af|\|_q \lesssim \|f\|_p \lesssim \delta^{\frac 3{2p}}$.
Taking $\delta \rightarrow 0$ yields $\frac 3p\le1+\frac1q$.

\smallskip

\noindent(c) $\frac1p\le\frac12$. Let $f=\chi_{B'}$ where 
\[
B'=\{x\in \R^4:|x_1|,|x_3|\le \delta ~\text{and}~ |x_2|,|x_4|\le1\}.
\]
Thus $\|f\|_p\sim \delta^{\frac 2p}$. Furthermore, $\sup_{t\in I}|\mathcal Af(x,t)|\gtrsim \delta$ for $x\in B^4(0,\epsilon)$ for $\epsilon>0$ small enough. Hence $\delta\lesssim \delta^{\frac 2p}$. Letting $\delta \rightarrow0$ proves part $(c)$.
\end{proof}

We prove the sharpness of the regularity order in Theorem \ref{LS}.

\begin{prop}
The estimates \eqref{est-LS} holds only if
\[
\gamma \ge \max \Big\{ \frac 2p-\frac 5q,~ \frac3p-\frac3q-1,  ~  
\frac 4p  -\frac2q -2 \Big\}.
\]
\end{prop}

The proof is almost identical to that of Proposition \ref{prop:nece-high} when $n=2$.

\begin{proof}
(a) $\gamma \ge \frac2p-\frac5q$. We choose $f=\chi_{\mathcal N_\delta(\Gamma)}$ as in part $(a)$ of Proposition \ref{prop:nece-high} and obtain
$\delta^{\frac5q}\le\|\mathcal Af\|_{L_{x,t}^q}\lesssim\|f\|_{L_\gamma^p}\lesssim \delta^{-\gamma+\frac2p}$. Hence $\gamma>\frac2p-\frac5q$ by taking $\delta \rightarrow 0$.

\smallskip
 
\noindent(b) $\gamma \ge -1-\frac3q+\frac3p$. Let $U_2$ and $V_2(t)$ be given in part $(b)$ of Proposition \ref{prop:nece-high}. We choose $f=\chi_{U_2}$. Then by a similar computation, $|\mathcal Af(x,t)|\gtrsim \delta$ for $x\in V_2(t)+O(\delta)$ for each fixed $t\in I$. Hence $\delta^{1+\frac3q}\le\|\mathcal Af\|_{L_{x,t}^q}\lesssim\|f\|_{L_\gamma^p}\lesssim \delta^{-\gamma+\frac3p}$. Taking $\delta\rightarrow0$ gives $\gamma \ge -1-\frac3q+\frac3p$.

\smallskip

\noindent(c) $\gamma \ge -2-\frac2q+\frac4p$. By the same example in part $(c)$ of Proposition \ref{prop:nece-high}, we see $|\mathcal Af(x,t)|\gtrsim\delta^2$ provided $x\in N_\delta(E_2(t))$  
for each fixed $t\in [1,2]$. Hence it follows that
$\delta^{2+\frac2q}\le\|\mathcal Af\|_{L_{x,t}^q}\lesssim\|f\|_{L_\gamma^p}\lesssim \delta^{-\gamma+\frac4p}$. Since $\delta$ is sufficiently small, we get $\gamma \ge \frac 4p-\frac 2q-2$ as  desired.
\end{proof}

\appendix
\section{Proof of Theorem \ref{thm:rank1}}\label{rank1case}
In this section, we deduce Theorem \ref{thm:rank1} by applying the local smoothing estimates for the wave operator obtained by \cite{Lee} to 
an averaging operator over a two-parameter family of surfaces by
\[ 
\mathfrak Af(x,t,s)=\int f\big(x-(tu,t\phi_1(u,v),sv,s\phi_2(v)\big)\varphi(u)\varphi(v)\,dudv.
\]

Since
$
\sup_{t\in I} |\mathcal Af(x,t)|\lesssim \sup_{s,t\in I} |\mathfrak A f(x,t,s)| ,
$
Theorem \ref{thm:rank1} is a consequence of the following. 
\begin{prop}\label{low-reg}
Suppose that $\phi_1$ and $\phi_2$ are smooth functions such that $\partial_u^2\phi_1\neq 0$ and $\partial_v^2\phi_2\neq0$.
Then
\[
\big\|\sup_{t,s \in I }|\mathfrak Af(\cdot,t,s)| \big\|_{L^{q }(\mathbb R^4)} \le C  
\|f\|_{L^{p }(\mathbb R^4)}
\]
for $(1/p,1/q) \in \mathcal W_1$.
Moreover, $\big\|\sup_{t,s \in I }|\mathfrak Af(\cdot,t,s)| \big\|_{L^{5,\infty}(\mathbb R^4)} \le C  
\|f\|_{L^{5/2,1}(\mathbb R^4)}$.
\end{prop}

In order to obtain the restricted weak type estimate at the endpoint $(1/p,1/q) = (2/5,1/5)$, we exploit the local smoothing estimate established by Lee \cite{Lee} instead of that by Guth--Wang--Zhang \cite{GWZ} because of $\epsilon$-loss on regularity. 
  
\begin{proof}
For $\xi' = (\xi_1,\xi_2)$ and $\xi'' =(\xi_3,\xi_4)$, we define the Littlewood-Paley projection operators by $\mathcal F( P_{j,k} f )(\xi) =\beta_{2^j}(  \xi') \beta_{2^k}(  \xi'') \widehat f(\xi)$ for $j,k \ge 1$.
Also, let us denote by $P_0'f$ and $P_0''f$ the Littlewood-Paley projection operator given by $\mathcal F (P_0'f) = \beta_0(\xi') \widehat f(\xi)$ and $\mathcal F (P_0''f) = \beta_0(\xi'') \widehat f(\xi)$, respectively.  

Since $ 1 = \beta_0  + \sum_{j \ge 1} \beta_{2^j}  $ for $\beta_0 = \sum_{j\le0} \beta_{2^j}$, 
we have 
\begin{align*}
1  
= \beta_0(\xi') + \sum_{j\ge 1} \beta_{2^j}(\xi') \beta_0(\xi'') + \sum_{j,k \ge 1} \beta_{2^j}(\xi')\beta_{2^k}(\xi'') .
\end{align*}
It follows that 
\begin{align}
\sup_{t,s} | \mathfrak A f(x,t,s)|    \le & \sup_{t,s} | \mathfrak A P_0' f(x,t,s)|  + \notag \\  &  +  \sum_{j\ge1}\sup_{t,s} | \mathfrak A P_{2^j, 0} f(x,t,s)|  +  \sum_{j,k\ge1} \sup_{t,s} | \mathfrak A P_{2^j, 2^k}f(x,t,s)| . \label{decomp_one}
\end{align}
  
It is obvious that $\sup_{t,s\in I} |\mathfrak A P_0' f(x,t,s)| $ is bounded by composition of the Hardy-Littlewood maximal function and the circular maximal function.
In fact, sine $P_0'$ is a convolution operator with a Schwartz function and $t\in I$, we have $\int P_0' f(x'-t(u,\phi_1(u,v)), x'' - s(v,\phi_2(v)))du$ is bounded by the Hardy-Littlewood maximal function $\mathfrak M f(x', x''-s(v,\phi_2(v)))$ for fixed $x''$, $s$, and $v$. 
Then, 
\[
\sup_{t,s} |\mathfrak A P_0' f(x,t,s)| \le \sup_s  \big| \int  \mathfrak M f(x', x''-s(v,\phi_2(v)))  dv \big|. 
\]
By the endpoint circular maximal theorem in \cite[Theorem 1.1]{Lee} and the $L^p$ boundedness of    $\mathfrak M$, the estimates in Proposition \ref{low-reg} hold for $\mathfrak A$ replaced by $ \mathfrak A P_0'$.

To handle the contribution from $\mathfrak AP_{j,k}f := \mathfrak A P_{2^j, 2^k}f$ in \eqref{decomp_one} (the other term can be handled in a similar way, so we omit the details), we apply \eqref{emb} to $F= \mathfrak A P_{j,k} f(x,\cdot)$ repeatedly in the $t$ and $s$ variables with $\ell=2^{-j}$ or $\ell=2^{-k}$, respectively. 
Then we obtain
\begin{align*}
\sup_{t,s\in I}| \mathfrak A P_{j,k} ( x,t,s)| 
&\le 2^{\frac jq+\frac k{q}}
\big\|\mathfrak A P_{j,k} f(x,\cdot)\big\|_{L_{t,s}^q}   
+2^{\frac jq-\frac k{q'}} 
\big\|\partial_s \mathfrak A P_{j,k} f(x,\cdot)\big\|_{L_{t,s}^q}   \\
&+2^{-\frac j{q'}+\frac kq}
\big\| \partial_t \mathfrak A P_{j,k} f(x,\cdot)\big\|_{L_{t,s}^q}  +
2^{-\frac j{q'}-\frac k{q'}} 
\big\| \partial_t \partial_s \mathfrak A P_{j,k} f(x,\cdot)\big\|_{L_{t,s}^q}.
\end{align*}

Observe that the amplitude of the multiplier associated with $\partial_t^\alpha \partial_s^\beta \mathfrak A P_{j,k}$ for $\alpha,\beta =0,1$, includes an factor $((u,\phi_1)\cdot \xi')^\alpha((v,\phi_2)\cdot \xi'')^\beta$, which is bounded by $2^{j\alpha}2^{k\beta}$. Since the amplitude of the multiplier and its derivatives are uniformly bounded, after factoring out $2^{j\alpha}2^{k\beta}$, it follows from the Mikhlin multiplier theorem that it suffices to prove the case $\alpha=\beta=0$.

First, we have the $L^2$-estimate
\begin{equation}\label{22-one}
\big\| \sup_{t,s\in I } | \mathfrak A P_{j,k} f(\cdot,t,s)| \big\|_{L^2 (\mathbb R^4)} \lesssim   \|f\|_{L^2}
\end{equation}
by Plancherel's theorem and the decay of the multipliers associated with $\mathfrak A P_{j,k}$.

Next, we claim that 
\begin{equation}\label{qp-one}
\big\|\sup_{t,s\in I } | \mathfrak A P_{j,k} f(\cdot,t,s)| \big\|_{L^q (\mathbb R^4)}
\lesssim 2^{(1-\frac 5q)(j+k)}  \|f\|_{L^p}.
\end{equation}
for $p,q$ satisfying $1/p + 3/q =1$ and $q >14/3$.  
The following sharp $L^p$--$L^q$ local smoothing estimate was established by Lee \cite[Proposition 1.2]{Lee}:
For $1/p + 3/q =1$ and $q >14/3$, the estimate
\begin{equation}\label{ls-wave}
\big\| e^{i t \sqrt{-\Delta} } P_j g \big\|_{L^q(\mathbb R^2\times I)} \le C 2^{(3/2 - 6/  q)j} \|g\|_{L^p(\mathbb R^2)}
\end{equation}
holds for $j\ge 0 $. 
Note that \eqref{ls-wave} is valid not only for the cone but also any conic extension of curves in $\mathbb R^2$ with nonvanishing curvature. 
Thus we have
\begin{align}\label{average_2}
\Big\| \int P_jg(x - t(u, \phi(u))du 
\Big\|_{L^q (\mathbb R^2\times I)} \le 2^{(1-6/q)j} \| g\|_{L^p(\R^2)},
\end{align}
where $\phi$ is a smooth function with nonvanishing curvature in $\R^2$.

By Minkowski's inequality and applying \eqref{average_2} repeatedly, we obtain
\begin{align*}
&\big\| \mathfrak A P_{j,k} f \big\|_{L^q (\mathbb R^4\times I^2)} \\
&\lesssim
\Big\|\int\Big\|\int P_{j,k}f(x'-t(u,\phi_1(u,v) ), x''-s(v,\phi_2(v)) )\,du\Big\|
_{L_{x',t}^q} \,dv\Big\|_{L_{x'',s}^q} \\
& \lesssim 2^{(1-\frac6q) j }
\Big \|\int \big\|  P_{j,k}f(x', x''-s(v,\phi_2(v) ) ) \big\|
_{L_{x'}^p} \,dv\Big\|_{L_{x'',s}^q} \\ 
&\lesssim 2^{(1-\frac6q) j } \Big\| \Big\|  \int     P_{j,k}f(x', x''-s(v,\phi_2(v) ) )   \, dv \Big\|_{L_{x'',s}^q}  \Big\|_{L_{x'}^p} \\ 
& \lesssim 2^{(1-\frac6q) (j+k) } \| f  \|_{L_{x}^p}  ,
\end{align*}
which implies \eqref{qp-one}. 

Following the argument in the proof of Theorem 1.1 in \cite{Lee}, 
we use Bourgain's interpolation trick (see Lemma \ref{Bourgain trick}) followed by interpolation between \eqref{qp-one} and \eqref{22-one} so that we obtain restricted weak type estimates 
\[
\big\| \sup_{t,s\in I} | \mathfrak A  f| \big\|_{L^{q,\infty} (\mathbb R^4)} \le   \|f\|_{L^{p,1}(\mathbb R^4)} 
\]
on the half-open line segment from $(1/2,1/2)$ to $(2/5,1/5)$, including $(2/5,1/5)$. 
This gives the desired estimate by interpolation with the trivial $L^\infty$-estimate. 
\end{proof}

\section*{Acknowledgements}
The authors would like to thank Prof. Sanghyuk Lee for helpful discussion and guidance.
This work was supported by the National Research Foundation of Korea (NRF) under grant number RS-2024-00342160 and RS-2025-24533387 (Ham), and G-LAMP RS-2024-00443714, RS-2024-00339824, JBNU research funds for newly appointed professors in 2024, and NRF2022R1I1A1A01055527 (Ko).

\bibliographystyle{plain}

\end{document}